\documentclass{amsart}

\usepackage[linecolor=black]{todonotes}
\usepackage{blindtext}
\usepackage{tcolorbox}
\usepackage{xypic}
\usepackage[pdf]{pstricks}
\usepackage{pst-plot}
\usepackage{url}

\usepackage{tikz} % for diagrams
\usetikzlibrary{matrix,arrows, calc, plotmarks}
\usetikzlibrary{decorations,decorations.text, decorations.pathmorphing}  % altermundus.com/pages/tkz/ornament/ 

\newtheorem{thm}{Theorem}[section]
\newtheorem{prop}[thm]{Proposition}
\newtheorem{cor}[thm]{Corollary}
\newtheorem{lemma}[thm]{Lemma}
\newtheorem{conj}[thm]{Conjecture}

\theoremstyle{definition}  % Bold headings and Roman body text.
\newtheorem{defn}[thm]{Definition}
\newtheorem{example}[thm]{Example}
\newtheorem{remark}[thm]{Remark}

\newcommand{\A}{\mathbb{A}}			% A affine 
\newcommand{\C}{\mathbb{C}}			% C complex numbers 
\newcommand{\F}{\mathbb{F}} 			% F field 
\newcommand{\Z}{\mathbb{Z}} 			% Z integers

\newcommand{\gammasmash}{\underset{\Gamma_\star S^0}{\wedge}}
\newcommand{\gammaotimes}{\underset{\Gamma_\star S^0}{\otimes}}

\newcommand{\map}{\rightarrow}
\renewcommand{\lto}{\rightarrow}

\newcommand{\tmf}{\mathit{tmf}}
\newcommand{\mmf}{\mathit{mmf}}
\newcommand{\Gast}{\Gamma_{\star}}			% Gamma_*
\newcommand{\Spt}{\textbf{Sp}} 				% Category of Spectra
\newcommand{\SptZ}{\Spt^{\Zop}}			% Z-diagrams of Spectra
\newcommand{\Xstdo}{X_{\star}^{\bullet}}
\newcommand{\Ystdo}{Y_{\star}^{\bullet}}
\newcommand{\Zop}{{\Z^{op}}}

\newcommand{\calC}{\mathcal{C}}
\newcommand{\calD}{\mathcal{D}}
\newcommand{\sslash}{/\kern -0.2em/}

\DeclareMathOperator{\colim}{colim} 			% colimit
\DeclareMathOperator{\Ext}{Ext}				% Ext
\DeclareMathOperator{\hocolim}{hocolim} 		% homotopy colimit
\DeclareMathOperator{\holim}{holim} 			% homotopy limit
\DeclareMathOperator{\id}{id}				% identity id
\DeclareMathOperator{\Map}{Map}			% Map
\DeclareMathOperator{\Mod}{Mod}			% Mod
\DeclareMathOperator{\Sq}{Sq}
\DeclareMathOperator{\Tot}{Tot} 			% Tot

\begin{document}

\title{$\C$-motivic modular forms}
\author{Bogdan Gheorghe}
\address{Max-Planck-Institut f\"ur Mathematik\\
53111 Bonn, Germany}
\email{gheorghebg@mpim-bonn.mpg.de}

\author{Daniel C.\ Isaksen}
\address{Department of Mathematics\\
Wayne State University\\
Detroit, MI 48202, USA}
\email{isaksen@wayne.edu}
\thanks{The second author was supported by NSF grant DMS-1606290.}

\author{Achim Krause}
\address{WWU M\"unster\\
Einsteinstr. 62, 48149 M\"unster, Germany}
\email{achim.t.krause@gmail.com}

\author{Nicolas Ricka}
\address{IRMA, 7 Rue Ren\'e Descartes, 67000 Strasbourg, France}
\email{n.ricka@unistra.fr}



\date{\today}

\begin{abstract}
We construct a topological model for cellular, $2$-complete, stable 
$\C$-motivic homotopy theory that uses no algebro-geometric foundations.
We compute the Steenrod algebra in this context, and we construct
a ``motivic modular forms" spectrum over $\C$.
\end{abstract}

\maketitle

\section{Introduction}

The topological modular forms spectrum $\tmf$ \cite{tmfbook}
detects a significant portion of the stable homotopy groups.
Classes detected by $\tmf$ are viewed as well-understood,
while classes not detected by $\tmf$ are viewed
as more exotic.  Similarly, Adams differentials, or Adams-Novikov
differentials, detected by $\tmf$ are well-understood.
From this perspective, the essential property of $\tmf$
is that its cohomology is the quotient $A\sslash A(2)$ of the
Steenrod algebra $A$ by the subalgebra $A(2)$ generated
by $\Sq^1$, $\Sq^2$, and $\Sq^4$.
This means that the cohomology of $A(2)$ is the $E_2$-page
of the Adams spectral sequence for $\tmf$, so this spectral
sequence can be computed effectively.

For the purposes of computing motivic stable homotopy groups,
it is desirable to have an analogous ``motivic modular forms"
spectrum $\mmf$ whose cohomology is $A\sslash A(2)$, in the motivic sense.
The existence of such a motivic spectrum over $\C$
immediately resolves
the status of many possible differentials in the
classical Adams spectral sequence \cite{IWX18}.
Such a motivic spectrum was first considered (hypothetically)
in \cite{Isaksen-A(2)}.

One of the main goals of this manuscript is to establish
the existence of $\mmf$ in the $\C$-motivic context.
Our techniques heavily depend on an algebraically closed
base of characteristic zero, but our intuition is that 
motivic modular forms ought to exist in much greater generality.

\begin{conj}
A motivic modular forms spectrum exists over any smooth base.
\end{conj}

One possible approach to constructing $\mmf$ is to follow the
classical construction of $\tmf$ in the motivic context.
This would require a careful understanding of motivic elliptic
cohomology theories and moduli spaces
of $E_\infty$-structures.  

We avoid these technical difficulties with a novel approach.
We first construct a category of $\Gast S^0$-modules that is
entirely topological in nature.  Foundationally, it depends
only on infinite sequences of classical spectra, with no reference
to smooth schemes, affine lines, $\A^1-0$, etc.

The category of $\Gast S^0$-modules is carefully engineered
to mimic the computational properties of $\C$-motivic
stable homotopy theory.  
We define an ``Eilenberg-Mac Lane" 
$\Gast S^0$-module $\Gast H\F_2$, and we compute that its
homotopy groups are of the form $\F_2[\tau]$, i.e., the same
as the motivic cohomology of $\C$.
We also directly compute the co-operations 
for $\Gast H\F_2$, i.e., the Steenrod algebra in the
context of $\Gast S^0$-modules, and we obtain an answer
that is identical to the $\C$-motivic Steenrod algebra.
We emphasize that our construction and our calculations are independent
of the much harder motivic calculations of Voevodsky \cite{VoeZ2}
\cite{Voered} \cite{VoeEM}.

The key idea relies on an observation from \cite{HKO11}
about the structure of the $\C$-motivic Adams-Novikov spectral
sequence (see also \cite[Chapter 6]{StableStems}).  
The weights in this spectral sequence follow a simple pattern.
The object $\Gast S^0$ is defined in such a way that its 
bigraded homotopy groups are computed by an identical spectral sequence.

Having constructed the category of $\Gast S^0$-modules, we study
the $\Gast S^0$-module $\Gast \tmf$, and we show that this object
has the desired computational properties, i.e., that its cohomology
is $A\sslash A(2)$ in the $\Gast S^0$-module context.

Finally, we prove that the homotopy category of
$\Gast S^0$-modules is equivalent to the 
$2$-complete cellular $\C$-motivic stable homotopy category.
Thus, $\Gast \tmf$ corresponds to some cellular $\C$-motivic
spectrum that deserves to be called $\mmf$.

Motivic homotopy theory has been at the center of a recent breakthrough
in the computation of stable homotopy groups of spheres 
\cite{StableStems} \cite{IWX18}.
From the perspective of that project,
the category of $\Gast S^0$-modules makes 
motivic homotopy theory no longer relevant.  
In particular, the stable homotopy group computations of
\cite{IWX18} no longer logically
depend on Voevodsky's computations of the motivic
cohomology of a point, nor on his computation of the motivic
Steenrod algebra.
For this reason, we completely avoid results and constructions
from motivic homotopy theory until Section \ref{sctn:compare}.

On the other hand, the theory of $\Gast S^0$-modules does not
make motivic homotopy theory obsolete.  
Since $\Gast S^0$-modules only captures the cellular motivic spectra,
it misses phenomena of arithmetic interest.  
Moreover, we currently cannot construct analogous models for
motivic homotopy theory over bases other than $\C$, although it
seems plausible that at least $\mathbb{R}$-motivic cellular
spectra have a topological model.

Our construction of $\Gast S^0$-modules could potentially
be generalized to other contexts.  
The complex cobordism spectrum $MU$ is built into the definitions
from the very beginning, but one might attempt to use other
cohomology theories in the same way.  Moreover, the basic
construction could be iterated to obtain interesting
multi-graded homotopy theories.

The cofiber of $\tau$ is a very interesting motivic spectrum
with curious properties \cite{GheCt} \cite{GWX}.
It would be interesting to study the cofiber of $\tau$
from the perspective of $\Gast S^0$-modules.
We do not carry out this investigation in this manuscript
because it is not central to our goal of constructing
$\mmf$.  On the other hand, it is possible that
$\Gast S^0$-modules are useful for understanding exotic
motivic periodicities \cite{GheKwn} \cite{Krause18},
and the cofiber of $\tau$ would play a critical role in that
pursuit.

The functor $\Gast$ provides a new tool for producing motivic spectra
from classical spectra.  If $X$ is a classical cell complex with
cells in only even dimensions, 
then $\Gast X$ is a motivic cell complex in which
$2k$-dimensional cells of $X$ correspond to $(2k,k)$-dimensional
cells of $\Gast X$.
The behavior of odd-dimensional cells under $\Gast$
is a bit more complicated.
Table \ref{tab:Gast} shows that many
of the motivic spectra commonly studied can be constructed with this
tool.
\begin{table}
\caption{Some values of $\Gast$
\label{tab:Gast}}
\begin{center}
\begin{tabular}{ll}
$X$ & $\Gast X$ \\
\hline
$H\F_2$ & $H\F_2$ \\
$H\Z$ & $H\Z$ \\
$KU$ & $KGL$ \\
$ku$ & $kgl$ \\
$KO$ & $KQ$ \\
$ko$ & $kq$ \\
$BP$ & $BPGL$ \\
$BP \langle n \rangle$ & $BPGL \langle n \rangle$ \\
$MU$ & $MGL$ \\
\end{tabular}
\end{center}
\end{table}
On the other hand, the functor $\Gast$
does not seem to interact well with
$\eta$-periodicization; see Remark \ref{rem:Gamma-eta}
for more discussion.

We mention recent work of Pstragowski \cite{Pstragowski18}
that constructs a topological model for 
cellular $\C$-motivic stable homotopy theory
using different foundational techniques.  We have not attempted
to compare our construction directly, but we suspect that such
a direct comparison is possible.
We also mention work of Heine \cite{Heine17} that is more formal.

\subsection{Organization}

We begin in Section \ref{sec:filteredspectra} 
with a discussion of a certain
diagram category of spectra.  In Section \ref{sctn:Gamma},
we define the functor $\Gast$ from ordinary spectra to
filtered spectra, and we develop the key properties of this functor
that allows for a rich computational theory
of $\Gast S^0$-modules.

In Section \ref{sctn:Steenrod-algebra}, 
we begin our computations by calculating
the Steenrod alegbra for $\Gast S^0$-modules.
We carry these calculations further in Section \ref{sctn:mmf}
when we study $\Gast \tmf$.

Finally, in Section \ref{sctn:compare}, we show that
the homotopy theory of $\Gast$-modules is equivalent to the
$2$-complete cellular $\C$-motivic stable homotopy category.

\subsection{Notation and conventions}

Throughout the article, we adopt the perspective of
$\infty$-stable categories for our homotopy theories.
We are working entirely in a $2$-complete setting.
All objects and categories are implicitly $2$-complete,
even though the notation will not reflect that assumption.

We use the symbol $\Z$ in two different ways.
Sometimes it is an indexing category for a diagram, in which
case we are thinking of $\Z$ as a partially ordered set
in the standard way.  Sometimes it is the coefficients
of a calculation, in which case it really means the 
$2$-adic integers $\Z_2$.  This latter usage is consistent
with our implicit use of $2$-completions.

The category $\Spt$ is the usual 
stable $\infty$-category of ($2$-complete) spectra.

We will use different gradings in different contexts.
For the sake of clarity, we adopt the following conventions:
\begin{enumerate}
\item
$\ast$ denotes an integer grading corresponding to the usual
grading on spectra.  For example, this grading is used for 
stable homotopy groups of classical spectra.
\item
$\star$ represents an indexing in $\Zop$.
This symbol is used exclusively for filtered spectra, i.e.,
functors from the nerve $N(\Zop)$ to ($2$-complete) spectra.
For example, the notation $X_\star$ indicates a functor
\[
X: N(\Zop) \map \Spt.
\]
\item
$\bullet$ denotes the cosimplicial degree of a cosimplicial object. 
\end{enumerate}

\subsection{Acknowledgements}

We thank Saul Glasman for some helpful hints about $E_\infty$-rings
in stable $\infty$-categories.  We also thank Piotr Pstragowski for 
suggesting the structure of some of the computational arguments in 
Section \ref{sctn:Steenrod-algebra}.

\section{Filtered spectra and a colocalization} \label{sec:filteredspectra}

In this section, we construct a 
closed symmetric monoidal stable $\infty$-category of
filtered spectra.
The objects that we are considering come entirely from the classical stable homotopy category.
The monoidal structure arises from Day convolution.

Note that we are using only a few formal properties of the category $\Spt$ of spectra.  Our construction could be generalized to filtered objects in any presentable closed symmetric monoidal stable $\infty$-category,
but we work only with $\Spt$ for the sake of simplicity.

\subsection{Filtered spectra}

\begin{defn}
The category of \emph{filtered spectra} is the 
$\infty$-category $\SptZ$ of $\infty$-functors from the $\infty$-nerve of the category $\Zop$, considered as a poset, to the category $\Spt$ of spectra. 
\end{defn}

The objects of $\SptZ$ are $\Zop$-diagrams in $\Spt$ up to coherent 
homotopy, in the following sense.
A filtered spectrum $X_\star$ is a sequence $\{ X_w\ |\ w \in \Zop \}$
of spectra,
together with a choice of morphisms $X_w \lto X_{w'}$, for all $w \geq w'$, and a choice of coherent homotopies relating compositions of the structure morphisms. 
We will typically denote a filtered spectrum by a sequence
\begin{equation*}
X_\star = \cdots \lto X_{1} \lto X_0 \lto X_{-1} \lto \cdots, 
\end{equation*}
in which the homotopies are implicit.
Similarly, we describe a morphism $X_\star \lto Y_\star$ 
in this category by a sequence 
$\{ X_w \lto Y_w\ |\ w \in \Zop \}$ of 
coherently homotopically 
compatible maps.

In $\SptZ$, the weak equivalences are exactly the pointwise weak equivalences, i.e., the maps $f \colon X_\star \lto Y_\star$ 
such that each component $f_w \colon X_w \lto Y_w$ induces an isomorphism on homotopy groups. In particular, an object $X_\star$ is contractible precisely when every spectrum $X_w$ is contractible. 
Moreover, homotopy limits and homotopy colimits
are computed pointwise in this diagram category.

In the category $\SptZ$, we have a bigraded family of 
sphere objects, as described in Definition \ref{defn:SptZ-spheres}.
We will show later in Lemma \ref{lemma:spheresgenerate}
that these spheres generate all filtered spectra in the appropriate
homotopical sense.

\begin{defn}
\label{defn:SptZ-spheres}
The sphere $S^{s,w}$
of bidegree $(s,w)$ is the object of $\SptZ$ defined by
\begin{equation*}
\cdots \lto \ast \lto \ast \lto S^s \stackrel{\id}{\lto} S^s 
\stackrel{\id}{\lto} \cdots,
\end{equation*}
where $S^{s,w}_v$ is $\ast$ in if $v > w$ and
$S^{s,w}_v$ is $S^s$ if $v \leq w$.
\end{defn}

Recall that such categories of functors canonically possess a closed symmetric monoidal structure, given by the Day convolution product (see \cite{Glas16}, \cite[Section 2.2.6]{HA}), which is induced by the monoidal structure of the source category. Explicitly, in filtered degree $w$ we have
\begin{equation} \label{eq:Dayconv}
(X_\star \otimes Y_\star)_w \simeq \hocolim_{i+j \geq w} X_i \wedge Y_j.
\end{equation}
The unit for this product is the sphere
$S^{0,0}$ of Definition \ref{defn:SptZ-spheres}.
Moreover, the convolution product commutes with homotopy colimits in each variable \cite[Lemma 2.13]{Glas16}.
We denote by $\Sigma^{s,w}$ the suspension endofunctor 
\begin{equation*}
S^{s,w} \otimes (-) \colon \SptZ \lto \SptZ.
\end{equation*}
More concretely, $\Sigma^{s,w} X_\star$ is the filtered spectrum
such that $(\Sigma^{s,w} X_\star)_v$ is
$\Sigma^s (X_{v-w})$.

The ``formal suspension" in $\SptZ$ is the same as $\Sigma^{1,0}$,
i.e., there is a cofiber sequence
\[
X_\star \lto * \lto \Sigma^{1,0} X_\star
\]
for every filtered spectrum $X_\star$.

The category $\SptZ$ is tensored and cotensored over spectra in the evident point-wise way.

\begin{defn}
\label{defn:colim-star}
Let $\colim_\star: \SptZ \lto \Spt$
be the functor that takes a filtered spectrum $X_\star$
to the spectrum $\colim_w X_w$.
\end{defn}

We are now ready to show that the bigraded spheres generate the category of filtered spectra. 

\begin{defn} \label{de:filteredstems}
The stable homotopy group $\pi_{s,w} X_\star$ is the
abelian group
$[ S^{s,w}, X_\star ]$.
\end{defn}

The direct sum $\pi_{\ast, \star} X$ of all
stable homotopy groups of $X$ is a bigraded abelian group.

\begin{remark}
\label{remark:pi_s,w}
It follows from a standard adjunction argument that 
$\pi_{s,w} X_\star$ is equal to the 
pointwise homotopy group $\pi_s X_w$.
\end{remark}

\begin{lemma} \label{lemma:spheresgenerate}
In the $\infty$-category $\SptZ$, 
a map is a weak equivalence if and only if 
it induces an isomorphism on $\pi_{\ast,\star}$. 
Equivalently, $\Spt^{\Zop}$ is generated under homotopy colimits by the spheres $S^{0,w}$ for $w \in \Zop$. 
\end{lemma}

\begin{proof}
By definition, a map $f \colon X_\star \lto Y_\star \in \SptZ$ is a weak equivalence if and only if
each $f_w \colon X_w \lto Y_w$
is a weak equivalence,
and $f_w$ is a weak equivalence if and only if 
it induces a $\pi_*$ isomorphism. 
Finally, Remark \ref{remark:pi_s,w} establishes the first claim.

The second claim follows from a standard argument (see for instance \cite[Theorem 1.2.1]{HPS}), since the first claim implies
that $X_\star$ is equivalent to the zero object in $\SptZ$ if and only if
the bigraded abelian group $\pi_{\ast, \star} X_\star$ is zero.
\end{proof}

\subsection{$t$-structure}
\label{subsctn:t-structure}

\begin{defn}
\label{defn:t-structure}
\mbox{}
\begin{enumerate}
\item
Let $\SptZ_{\geq 0}$ be the full subcategory of
$\SptZ$ consisting of filtered spectra $X_\star$ such that
$\pi_{s,w} X_\star = 0$ for $s < 2w$, i.e., if
$\pi_s X_w = 0$ for $s < 2w$.
\item
Let $\SptZ_{\leq 0}$ be the full subcategory of
$\SptZ$ consisting of filtered spectra $X_\star$ such that
$\pi_{s,w} X_\star = 0$ for $s > 2w$, i.e., if
$\pi_s X_w = 0$ for $s > 2w$.
\end{enumerate}
\end{defn}

\begin{prop}
\label{prop:t-structure}
Definition \ref{defn:t-structure} equips $\SptZ$ with a
t-structure.
\end{prop}

\begin{proof}
Suppose that $X_\star$ and $\Sigma Y_\star$ belong to
$\SptZ_{\geq 0}$ and $\SptZ_{\leq 0}$ respectively.
We will show that any map $X_\star \lto Y_\star$ 
is trivial.
The simplicial mapping space $\Map(X_\star, Y_\star)$
is the homotopy limit
\[
\holim_{i \geq j} \Map(X_i, Y_j).
\]
When $i \geq j$, the space $\Map(X_i, Y_j)$ is contractible since
$\pi_n X_i = 0$ if $n < 2 i$ and $\pi_n Y_j = 0$ if $n \geq 2 j$.
Therefore, the homotopy limit is contractible as well.

The subcategory $\SptZ_{\geq 0}$ is closed under suspension
since the suspension functor on filtered spectra is defined 
pointwise.
Similarly, the subcategory
$\SptZ_{\leq 0}$ is closed under desuspension.

Finally, for any filtered spectrum $X_\star$, we have the diagram
\[
\xymatrix{
\ar[r] & \tau_{\geq 4} X_2 \ar[r] \ar[d] & 
\tau_{\geq 2} X_1 \ar[r] \ar[d] & \tau_{\geq 0} X_0 \ar[r] \ar[d] &
\tau_{\geq -2} X_{-1} \ar[r] \ar[d] & \tau_{\geq -4} X_{-2} \ar[r]\ar[d]&
\\
\ar[r] & X_2 \ar[r] \ar[d] & X_1 \ar[r] \ar[d] & 
X_0 \ar[r] \ar[d] & X_{-1} \ar[r] \ar[d] & X_{-2} \ar[r]\ar[d] &
\\
\ar[r] & \tau_{< 4} X_2 \ar[r] & 
\tau_{< 2} X_1 \ar[r] & \tau_{< 0} X_0 \ar[r] &
\tau_{< -2} X_{-1} \ar[r] & \tau_{< -4} X_{-2} \ar[r] &
,
}
\]
where $\tau_{\geq n}$ and $\tau_{< n}$ are the usual connective cover
and Postnikov section functors on spectra.
The top row of the diagram is a filtered spectrum
in $\SptZ_{\geq 0}$, the bottom row is a filtered
spectrum in $\SptZ_{\leq -1}$, and the columns are fiber sequences
of ordinary spectra.
Since fiber sequences of filtered spectra
are defined pointwise,
this diagram defines a fiber sequence of filtered spectra.
\end{proof}

\begin{lemma}
\label{lem:t-structure-smash}
The subcategory $\SptZ_{\geq 0}$ is closed under Day convolution.
\end{lemma}

\begin{proof}
Let $X_\star$ and $Y_\star$ belong to $\SptZ_{\geq 0}$.
Then $(X_\star \otimes Y_\star)_w$ is the cofiber of the map
\[
\bigvee_{a+b = w+1} X_a \wedge Y_b \lto
\bigvee_{i+j = w}  X_i \wedge Y_j.
\]
The homotopy groups of $X_i \wedge Y_j$ vanish in degrees
less than $2w$ because the homotopy groups of $X_i$ and of $Y_j$
vanish in degrees less than $2i$ and $2j$ respectively.
Similarly the homotopy groups of $X_a \wedge Y_b$
vanish in degrees less than $2(w+1)$.
By inspection of the long exact sequence in homotopy,
it follows that the homotopy groups of
$(X_\star \otimes Y_\star)_w$ also vanish in 
degrees less than $2w$.
\end{proof}

We write $(\tau_{\geq 0})_\star$ for the right adjoint, 
in the appropriate
$\infty$-categorical sense, to the inclusion
$\SptZ_{\geq 0} \lto \SptZ$.
More concretely, $(\tau_{\geq 0})_\star : \SptZ \lto \SptZ_{\geq 0}$ 
is the functor that takes a filtered spectrum $X_\star$ to 
\[
\xymatrix@1{
\cdots \ar[r] & \tau_{\geq 4} X_2 \ar[r] & 
\tau_{\geq 2} X_1 \ar[r] & \tau_{\geq 0} X_0 \ar[r] &
\tau_{\geq -2} X_{-1} \ar[r] & \tau_{\geq -4} X_{-2} \ar[r] &
\cdots
}
\]

\begin{cor}
\label{cor:t-structure-smash}
The truncation functor $(\tau_{\geq 0})_\star$ is
lax symmetric monoidal.
\end{cor}

\begin{proof}
Lemma \ref{lem:t-structure-smash}
says that the inclusion $\SptZ_{\geq 0} \lto \SptZ$
is strong symmetric monoidal.
Therefore, its right adjoint is lax symmetric monoidal.
\end{proof}

\section{The functor $\Gamma_\star$}
\label{sctn:Gamma}

The goal of this section is to define and study a functor
$\Gamma_\star$ from $\Spt$ to $\SptZ$.
The functor is constructed in such a way that it interacts
with the Adams-Novikov spectral sequence in an interesting way.

\begin{defn}
Let $MU^{\bullet +1}$ be the cosimplicial spectrum 
whose $n$th term is the $(n+1)$-fold smash product of $MU$ with
itself, and whose
faces and degeneracies are induced from the multiplication and unit 
of the ring structure on $MU$.
\end{defn}

The cosimplicial spectrum $MU^{\bullet + 1}$ is the usual
tool for constructing the $MU$-based Adams spectral sequence.
More precisely, the Bousfield-Kan spectral sequence that computes
the totalization $\Tot ( X \wedge MU^{\bullet + 1})$ is the
Adams-Novikov spectral sequence that converges to the homotopy of $X$.
Convergence of the Adams-Novikov spectral sequence
just means that 
the totalization $\Tot (X \wedge MU^{\bullet + 1})$
is equivalent to the (appropriately completed) spectrum $X$.

We work with the $MU$-based Adams spectral sequence, rather than
the $BP$-based Adams spectral sequence, because $MU$ has better
formal multiplicative properties than $BP$.  Computationally,
there is no difference between the two perspectives.

\begin{defn}
Let $X$ be a spectrum.
\begin{enumerate}
\item
Let 
$\tau_{\geq 2w} \left( X \wedge MU^{\bullet +1} \right)$ be
the cosimplicial spectrum 
formed by taking the $(2w-1)$-connective cover 
of each term of the cosimplicial spectrum
$X \wedge MU^{\bullet + 1}$.
\item
Define the filtered spectrum $\Gamma_\star X $ such that
each $\Gamma_w X$ is
the totalization
$\Tot \tau_{\geq 2w} \left( X \wedge MU^{\bullet +1} \right)$.
\end{enumerate}
\end{defn}

We now turn to a particular spectral sequence computing the homotopy groups of $\Gamma_\star X$.
We grade $\Ext$ in the form $(s,f)$,
where $f$ is the homological degree and $s+f$ is the total degree.
In other words, $s$ is the stem, and $f$ is the Adams-Novikov filtration.
While inconsistent with historical notation,
this choice aligns well with the usual graphical representations
of the Adams-Novikov spectral sequence.

\begin{prop} \label{pro:anssgasts}
Let $X$ be a spectrum, and let $w \in \Zop$. 
There is a spectral sequence
\begin{equation*}
E_2^{s,f,w}(X) = \left\lbrace 
\begin{matrix} \Ext_{MU_*MU}^{s,f}(MU_* ,MU_* X)  
& \text{ if } s+f \geq 2w \\
0 & \text{ otherwise } \end{matrix} \right.
\end{equation*}
that converges to $\pi_{s,w} \Gamma_\star X$.
These spectral sequences are compatible with the filtered
spectrum structure on $\Gamma_\star X$, 
in the sense that 
the morphism $\Gamma_w X \lto \Gamma_{w-1} X$
induces the evident inclusion 
$E_2^{s,f,w}(X) \lto E_2^{s,f,w-1}(X)$.
\end{prop}

\begin{proof}
The spectral sequence is the Bousfield-Kan spectral sequence associated to the cosimplicial spectrum 
$\tau_{\geq 2w} (X \wedge MU^{\bullet +1})$.
The spectral sequence converges to
$\pi_{\ast, \star} \Gamma_\star X$ by definition of $\Gamma_\star X$.
Our goal is to identify the $E_2$-page of this
Bousfield-Kan spectral sequence.

The $E_1$-page of the Bousfield-Kan spectral sequence is
\[
E_1^{s,f,w} = \pi_{s+f}\left(\tau_{\geq 2w} MU^{\wedge f+1} \right), 
\]
and the $d_1$ differential is the map
\[
\pi_{s+f}  \left( \tau_{\geq 2w} (X \wedge MU^{\wedge f+1} ) \right)
\lto \pi_{s+f} \left( \tau_{\geq 2w} ( X \wedge MU^{\wedge f+2} )\right)
\]
induced by the alternating sum
of cofaces. 
When $s+f < 2w$, the $E_1$-page is zero.
When $s+f \geq 2w$, the $E_1$-page is isomorphic to
\[
E_1^{s,f,w}(X) = \pi_{s+f} \left( X \wedge MU^{\wedge f+1} \right),
\]
and the $d_1$ differential is the map
\[
\pi_{s+f} \left( X \wedge MU^{\wedge f+1} \right)
\lto \pi_{s+f} \left( X \wedge MU^{\wedge f+2} \right).
\]
This is identical to the $E_1$-page and $d_1$ differential
for the classical Adams-Novikov spectral sequence for $X$.
Therefore, $E_2^{s,f,w}(X)$ is equal to the 
Adams-Novikov $E_2$-page for $X$ when $s+f \geq 2w$.
\end{proof}

\begin{remark}
For fixed $w$, the spectral sequence $E_2^{s,f,w}(X)$
is a truncated version of the classical Adams-Novikov spectral
sequence for $X$, as shown in Figure \ref{fig:ANSS-truncated}.  
Above the line $s+f = 2w$ of slope $-1$,
$E_2^{s,f,w}(X)$ equals the classical Adams-Novikov spectral
sequence.  Below this line, $E_2^{s,f,w}(X)$ is zero.
If $X$ is $n$-connected and $2w \leq n$, then
then the truncation is trivial,
and $E_2^{s,f,w}(X)$ is equal to the classical Adams-Novikov
spectral sequence.

Classical Adams-Novikov differentials whose source is above
the line of slope $-1$ occur as well in
$E_2^{s,f,w}(X)$.
However, classical Adams-Novikov differentials whose source
is below the line of slope $-1$ do not occur in $E_r^{s,f,w}(X)$.
Of particular interest are classical differentials that
cross the line $s+f= 2w$.  Consequently, there are non-zero
permanent cycles in $E_\infty^{s,f,w}(X)$ that are zero in the
classical Adams-Novikov $E_\infty$-page.
\end{remark}

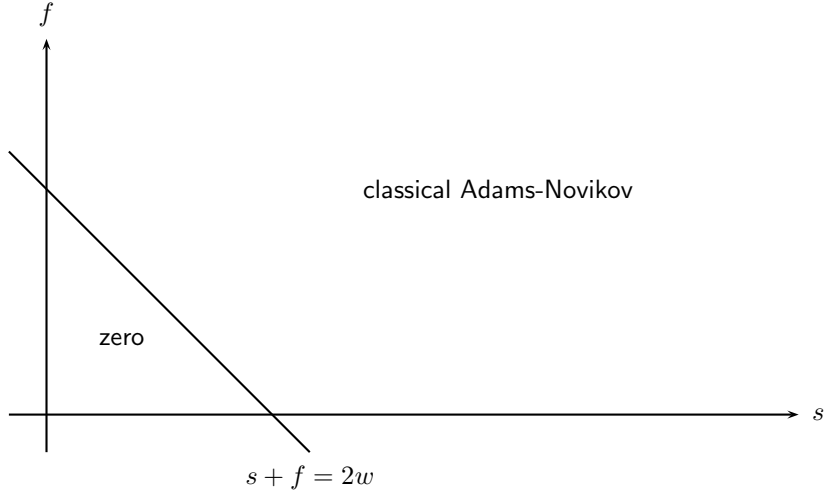
\begin{figure}
\caption{
The spectral sequence of Proposition \ref{pro:anssgasts}.
The region below the diagonal line is zero.
The region above the diagonal line is isomorphic to the
classical Adams-Novikov spectral sequence in that range.
\label{fig:ANSS-truncated}
}
\begin{pspicture}(-1,-1)(10,6)
\psline{->}(0,-0.5)(0,5)
\psline{->}(-0.5,0)(10,0)
\psline(-0.5,3.5)(3.5,-0.5)
\uput[0](10,0){$s$}
\uput[90](0,5){$f$}
\uput[-90](3.5,-0.5){$s + f = 2w$}
\rput(1,1){\textsf{zero}}
\rput(6,3){\textsf{classical Adams-Novikov}}
\end{pspicture}
\end{figure}

\begin{remark}
Not coincidentally, when $X = S^0$, the spectral sequence of 
Proposition \ref{pro:anssgasts} is identical to the motivic
Adams-Novikov spectral sequence for computing the
stable homotopy groups of the motivic sphere spectrum 
\cite{HKO11} \cite{StableStems}.
Consequently, the bigraded homotopy groups of the filtered
spectrum $\Gamma_\star S^0$ are the same as the bigraded
motivic homotopy groups of the motivic sphere spectrum.
\end{remark}

Later when we study specific examples of $\Gast X$, we will use
specific computational information about $X$ to deduce
analogous computational information about $\Gast X$.
Proposition \ref{prop:colim-Gamma} is the precise tool
for transporting such information.

\begin{prop}
\label{prop:colim-Gamma}
Let $X$ be a bounded-below spectrum. Then
$\colim_\star \Gast X$ is equivalent to $X$.
\end{prop}

\begin{proof}
By definition, $\colim_\star \Gast X$ equals
\[
\colim_w \Tot \left (\tau_{\geq 2 w} 
\left( X \wedge MU^{\bullet + 1} \right) \right).
\]
Since $X$ is bounded below, we can write the $\Tot$ as an inverse limit of $\Tot^n$, where the maps $\Tot^n \to \Tot^{n-1}$ increase in connectivity. Since filtered colimits commute with the finite limits $\Tot^n$, and preserve connectivity, one sees that the filtered colimit commutes with totalization in the present situation. So we have
\[
\Tot \left ( \colim_w \tau_{\geq 2 w} 
\left( X \wedge MU^{\bullet + 1} \right) \right) =
\Tot
\left( X \wedge MU^{\bullet + 1} \right).
\]
The latter object is the same as $X$, precisely because
the Adams-Novikov spectral sequence converges for connective $X$.
\end{proof}

\subsection{The ring object $\Gamma_\star S^0$}

Our next goal is to show that $\Gamma_\star S^0$ is
a ring object in the category $\SptZ$.
This will allow us to consider $\Gamma_\star S^0$-modules
in $\SptZ$.

We will need to work in the category of cosimplicial filtered
spectra, or equivalently in filtered cosimplicial spectra.
Such objects are $(\Zop \times \Delta)$-shaped
diagrams of spectra.

In accordance with our general grading conventions, 
$\Xstdo$ denotes a cosimplicial filtered spectrum, 
where $\star$ refers to the filtered degree 
while $\bullet$ refers to the cosimplicial degree. 
Thus every $X_{\star}^s$ is a filtered spectrum, and
every $X^\bullet_w$ is a cosimplicial spectrum. 

The category of cosimplicial filtered spectra is symmetric
monoidal with respect 
to Day convolution applied pointwise in the cosimplicial
direction.
More concretely, if 
$\Xstdo$ and $\Ystdo$ are cosimplicial filtered spectra, then
$\Xstdo \otimes \Ystdo$ is the cosimplicial
filtered spectrum such that
\[
\left( \Xstdo \otimes \Ystdo \right)^s_\star =
X^s_{\star} \otimes Y^s_{\star},
\]
where the latter product is the Day convolution product 
of filtered spectra.
We refer to this product
in the proof of Proposition \ref{prop:Gamma-lax-monoidal}
when we study lax symmetric monoidal functors taking values
in cosimplicial filtered spectra.

\begin{prop}
\label{prop:Gamma-lax-monoidal}
The functor $\Gamma_\star$ is lax symmetric monoidal.
\end{prop}

\begin{proof}
We will show that $\Gamma_\star$ is a composition
of three lax symmetric monoidal functors.

First, consider the functor
$\Spt \lto \Spt^\Delta$ that takes
a spectrum $X$ to the cosimplicial spectrum 
$X \wedge MU^{\bullet + 1}$.
This functor is lax symmetric monoidal, i.e., there are natural maps
\[
(X \wedge MU^{\bullet + 1} ) \wedge
(Y \wedge MU^{\bullet + 1} ) \lto
(X \wedge Y) \wedge MU^{\bullet + 1},
\]
because $MU^{\bullet + 1}$ is a commutative ring object 
in the category of cosimplicial spectra.

Second, recall the truncation functor
$(\tau_{\geq 0})_\star$ on filtered spectra
defined at the end of 
Section \ref{subsctn:t-structure}.
Let $(\tau_{\geq 0})^\bullet_\star$ be the functor
from cosimplicial filtered spectra to cosimplicial filtered
spectra
that applies $(\tau_{\geq 0})_\star$ pointwise in the 
cosimplicial direction.  
This functor is also lax symmetric monoidal because
$(\tau_{\geq 0})_\star$ is lax symmetric monoidal by 
Corollary \ref{cor:t-structure-smash}.

Finally, let $\Tot_\star$
be the functor from cosimplicial filtered spectra
to filtered spectra that applies $\Tot$ pointwise
in the filtered direction.
This functor is lax symmetric monoidal because it is a homotopy limit
in the category of filtered spectra.
\end{proof}

\begin{remark}
The functor $\Gamma_\star$ is not strong monoidal.
We will see in Examples \ref{ex:Gamma-S^2} and \ref{ex:Gamma-S^1} that
$\Gamma_\star S^2$ equals $\Sigma^{2,1} \Gamma_\star S^0$,
while
$\Gamma_\star S^1$ equals $\Sigma^{1,0} \Gamma_\star S^{0}$.
Therefore, the map
\[
\Gamma_\star S^1 \gammasmash \Gamma_\star S^1 \lto
\Gamma_\star(S^1 \wedge S^1)
\]
is the map $\tau: \Sigma^{2,0} \Gamma_\star S^0 \lto
\Sigma^{2,1} \Gamma_\star S^0$,
which is not an equivalence of filtered spectra.
\end{remark}

\begin{thm} \label{thm:sering}
The filtered spectrum $\Gamma_\star S^0$ is an $E_\infty$-ring
object in the category of filtered spectra.
\end{thm}

\begin{proof}
This follows immediately from Proposition \ref{prop:Gamma-lax-monoidal}
because 
lax symmetric mon\-oidal functors preserve $E_\infty$-ring objects.
\end{proof}

We can now define the category in which we are primarily interested.

\begin{defn}
\label{defn:Gamma-mod}
Let $\Mod_{\Gamma_\star S^0}$
be the category of left $\Gamma_\star S^0$-modules 
in the category of filtered spectra.
\end{defn}

We showed in Proposition \ref{prop:Gamma-lax-monoidal} 
that the functor $\Gamma_\star$ is lax symmetric monoidal.
Therefore, $\Gamma_\star X$ is a $\Gamma_\star S^0$-module
for every spectrum $X$.

Equivalences in $\Mod_{\Gamma_\star S^0}$ are defined to be
equivalences on the underlying filtered spectra.
For any two $\Gast S^0$-modules $X$ and $Y$,
let 
$[ X , Y ]_{\Gast S^0}$ be the set of 
homotopy classes of $\Gast S^0$-module maps from $X$ to $Y$.

\begin{prop}
\label{prop:Gamma-generate}
In the $\infty$-category $\Mod_{\Gamma_\star S^0}$,
a map is a weak equivalence if and only if 
it induces an isomorphism on 
$[ \Sigma^{p,q} \Gast S^0, - ]_{\Gast S^0}$ for all 
$p$ and $q$.
Equivalently, $\Mod_{\Gamma_\star S^0}$ is generated under homotopy colimits by the objects $\Sigma^{p,q} \Gast S^0$ for all $p$ and $q$.
\end{prop}

\begin{proof}
This follows from Lemma \ref{lemma:spheresgenerate}, together 
with the adjunction
\[
[ \Sigma^{p,q} \Gast S^0, X ]_{\Gast S^0} \cong
\pi_{p, q} X.
\]
\end{proof}

\subsection{Exactness properties of $\Gamma_\star$}

In general, the functor $\Gamma_\star$ is not exact, in the sense
that it does not preserve all cofiber sequences.
However, we shall show that $\Gamma_\star$ preserves 
certain types of cofiber sequences.
These results are essential for
computations later in Sections \ref{sctn:Steenrod-algebra} 
and \ref{sctn:mmf}.

\begin{lemma}
\label{lem:Gamma-even-suspend}
For any spectrum $X$,
the filtered spectrum
$\Gamma_\star(\Sigma^{2k} X)$ is equivalent to
$\Sigma^{2k,k} \Gamma_\star X$.
\end{lemma}

\begin{proof}
We have that
$\Gamma_w (\Sigma^{2k} X)$ is equal to
\[
\Tot \left( \tau_{\geq 2w}  
\left( S^{2k} \wedge X \wedge MU^{\bullet + 1} \right) \right),
\]
which is equivalent to
\[
\Tot \left( S^{2k} \wedge \tau_{\geq 2(w-k)} 
\left(X \wedge MU^{\bullet + 1} \right) \right).
\]
The functor $\Tot$ commutes (up to homotopy) with suspension
since homotopy limits commute with desuspension.
We conclude that $\Gamma_w (\Sigma^{2k} X)$ is equal
to $\Sigma^{2k} \Gamma_{w-k} X$.
\end{proof}

\begin{example}
\label{ex:Gamma-S^2}
When $X$ is $S^0$,
Lemma \ref{lem:Gamma-even-suspend}
implies that
$\Gamma_\star S^{2k}$ equals
$\Sigma^{2k,k} \Gamma_\star S^0$.
\end{example}

\begin{lemma}
\label{lem:Gamma-odd-suspend}
Let $X$ be a spectrum such that
$MU_* X$ is concentrated in even degrees.
Then
$\Gamma_\star(\Sigma^{2k+1} X)$ equals
$\Sigma^{2k+1,k} \Gamma_\star X$.
\end{lemma}

\begin{proof}
We have that
$\Gamma_w (\Sigma^{2k+1} X)$ is equal to
$\Tot \left( \tau_{\geq 2w}  
\left( S^{2k+1} \wedge X \wedge MU^{\bullet + 1} \right) \right)$,
which is equivalent to
$\Tot \left( S^{2k+1} \wedge \tau_{\geq 2(w-k)} 
\left(X \wedge MU^{\bullet + 1} \right) \right)$
because
each $X \wedge MU^{\bullet + 1}$
has homotopy groups concentrated in even degrees.
Similarly to the proof of Lemma \ref{lem:Gamma-even-suspend},
we conclude that
$\Gamma_w (\Sigma^{2k+1} X)$ is equal
to $\Sigma^{2k+1} \Gamma_{w-k} X$.
\end{proof}

\begin{example}
\label{ex:Gamma-S^1}
When $X$ is $S^0$,
Lemma \ref{lem:Gamma-odd-suspend}
implies that the filtered spectrum
$\Gamma_\star S^{2k+1}$ equals
$\Sigma^{2k+1,k} \Gamma_\star S^0$.
\end{example}

\begin{remark}
\label{rem:Gamma-eta}
The functor $\Gast$ does not commute with suspensions.
For example, consider the first Hopf map $\eta: S^1 \lto S^0$.
Then $\Gast \eta$ is a map $\Sigma^{1,0} \Gast S^0 \lto \Gast S^0$,
of relative degree $(1,0)$.
On the other hand, consider $\Sigma \eta: S^2 \lto S^1$.
Then $\Gast \left( \Sigma \eta \right)$ is a map 
$\Sigma^{2,1} \Gast S^0 \lto \Sigma^{1,0} \Gast S^0$,
of relative degree $(1,1)$.
\end{remark}

\begin{prop}
\label{prop:Gamma-exact}
Let $X \lto Y \lto Z$ be a cofiber sequence such that
$MU_{2w-1} X \lto MU_{2w-1} Y$ is injective for all $w$.
Then
\[
\Gamma_\star X \lto \Gamma_\star Y \lto \Gamma_\star Z
\]
is a cofiber sequence of filtered spectra.
\end{prop}

\begin{proof}
The given condition implies that 
\[
\pi_{2w-1} \left(X \wedge MU^{n+1}\right) \lto 
\pi_{2w-1} \left(Y \wedge  MU^{n+1}\right)
\]
is injective for all $w$ and all $n$.
Since the composite
\begin{equation}
\tau_{\geq 2w} \left (X \wedge MU^{n + 1} \right) \lto
\tau_{\geq 2w} \left (Y \wedge MU^{n + 1} \right) \lto
\tau_{\geq 2w} \left (Z \wedge MU^{n + 1} \right)
\label{eq:cofibersequence}
\end{equation}
is nullhomotopic, we get a map
\[
\operatorname{cofib}\left(\tau_{\geq 2w} \left (X \wedge MU^{n + 1} \right) 
\lto \tau_{\geq 2w} \left (Y \wedge MU^{n + 1} \right)\right) 
\lto \tau_{\geq 2w} \left (Z \wedge MU^{n + 1} \right).
\]
A diagram chase in homotopy groups shows that this is an equivalence, so \eqref{eq:cofibersequence} is a cofiber sequence
for all $w$ and all $n$.
The functor $\Tot$ preserves cofiber sequences
because fiber sequences are the same as cofiber sequences,
and $\Tot$ is a homotopy limit.
Therefore,
\[
\Gamma_w X \lto \Gamma_w Y \lto \Gamma_w Z
\]
is a cofiber sequence for all $w$.
\end{proof}

Proposition \ref{prop:Gamma-exact}
includes a technical condition about odd $MU$-homology.
We would like to show that this condition holds for a large
class of spectra.  With that goal in mind,
Definition \ref{defn:even-cell}
encapsulates some standard notions in a convenient form.

\begin{defn}
\label{defn:even-cell}
A spectrum $X$ is a \emph{bounded below, finite type, even-cell
complex} if it is a finite complex with cells only in even dimensions, 
or if it is the homotopy colimit of a sequence
\[
* = X^{(0)} \lto X^{(1)} \lto X^{(2)} \lto \cdots,
\]
where there are cofiber sequences
\[
X^{(n-1)} \lto X^{(n)} \lto S^{2k_n}
\]
for some integers $k_n$ that tend to $\infty$ as $n$ approaches
$\infty$.
\end{defn}

\begin{example}
\label{ex:even-cell}
Suppose that $X$ is a bounded below spectrum.
Then $X$ has, up to $p$-completion, a cell structure in which the cells are in 
one-to-one correspondence with a basis for $H_*(X; \F_p)$.
If $H_*(X;\F_p)$ is concentrated in even degrees, and each
$H_n(X;\F_p)$ is finite-dimensional, then $X$ is
a bounded below, finite type, even-cell complex. Similar statements hold (without any completion) for integer coefficient homology, provided $H_*(X;\Z)$ is free over $\Z$.

Specific examples include $BP$, $BP\langle n \rangle$,
and $MU$.
\end{example}

Note that if 
$X$ and $Y$ are bounded below, finite type, even-cell complexes,
then so is $X \wedge Y$.
This follows from the standard fact that
$X \wedge Y$ has a cell structure in which the $n$-cells correspond
to pairs of $i$-cells in $X$ and $j$-cells in $Y$ such that $i + j = n$.

\begin{lemma}
\label{lem:MU_*X-even}
Let $X$ be a bounded below, finite type, even-cell complex. Then $MU\wedge X$ splits as an $MU$-module into a wedge $\bigvee \Sigma^{2k_n} MU$ of even shifts of $MU$. In particular, if
$Y$ is a spectrum whose $MU$-homology is bounded below and
concentrated in even degrees,
then $MU_* (X \wedge Y)$ is concentrated in even degrees.
\end{lemma}

\begin{proof}
Let $X$ be the homotopy colimit of the diagram
\[
X^{(0)} \lto X^{(1)} \lto X^{(2)} \lto \cdots,
\]
with cofiber sequences
\[
X^{(n-1)} \lto X^{(n)} \lto S^{2k_n}.
\]
Inductively assume that the $MU$-homology of $X^{(n-1)}$ is concentrated in even degrees. Since the $MU$-homology of $S^{2k_n}$ is also concentrated in even degrees, we get a short exact sequence
\[
0 \to MU_* X^{(n-1)} \lto MU_* X^{(n)} \lto MU_* S^{2k_n} \lto 0.
\]
Since the $MU$-homology of $S^{2k_n}$ is furthermore free as an $MU_*$-module, the sequence splits. Inductively, we see that $MU_* X^{(n)}$ is free as an $MU_*$-module. Since $MU_*$ commutes with filtered colimits, this follows for $MU_* X$ as well.
Finally, a basis $x_{2k_n}$ for $MU_* X$ gives rise to a map
\[
\bigvee \Sigma^{2k_n} MU \to MU \wedge X,
\]
which is an equivalence since it is an isomorphism on homotopy groups.
For the other statement, observe that
\[
MU_*(X\wedge Y) = \pi_*(MU\wedge X \wedge Y) \simeq \bigoplus \pi_*(\Sigma^{2k_n}MU \wedge Y) \simeq \bigoplus MU_{*-2k_n} Y,
\]
which is concentrated in even degrees.
\end{proof}

\begin{remark}
When $Y$ is $S^0$, Lemma \ref{lem:MU_*X-even} shows
that the $MU$-homology of a 
bounded below, finite type, even-cell complex is
concentrated in even degrees.
\end{remark}

\begin{cor}
\label{cor:Gamma-exact}
Let 
\[
X \lto Y \lto Z
\]
be a cofiber sequence of bounded below, finite type,
even-cell complexes, and let $W$ be a spectrum whose
$MU$-homology is bounded below and concentrated in even degrees.
Then
\[
\Gamma_\star \left( X \wedge W \right) \lto 
\Gamma_\star \left( Y \wedge W \right) \lto 
\Gamma_\star \left( Z \wedge W \right)
\]
is a cofiber sequence.
\end{cor}

\begin{remark}
When $W$ is $S^0$,
Corollary \ref{cor:Gamma-exact} shows that $\Gast$
preserves cofiber sequences of bounded below, finite type,
even-cell complexes.
\end{remark}

\begin{proof}
Lemma \ref{lem:MU_*X-even} establishes the hypothesis of
Proposition \ref{prop:Gamma-exact}.
\end{proof}

\begin{lemma}
\label{lem:Gamma-seq-colimit}
Let 
\[
\cdots\lto X_i \lto X_{i+1} \lto \cdots
\]
be a sequential diagram of uniformly bounded below spectra such that the connectivity of the maps $X_i \to X_{i+1}$ tends to infinity. Let $X = \colim_i X_i$.
Then
\[
\hocolim_n \Gamma_\star \left( X_i\right) \lto 
\Gamma_\star \left( X\right)
\]
is an equivalence.
\end{lemma}

\begin{proof}
We observe that formation of the cosimplicial object $\tau_{\geq 2w}(X_i\wedge MU^{\bullet+1})$ commutes with filtered colimits levelwise. We have to check that in our given situation, the filtered colimit also commutes with totalisation. To see this, we first recall that filtered colimits always commute with finite limits. Thus, in the following diagram
\begin{center}
\begin{tikzpicture}
\matrix (m) [matrix of math nodes, row sep=2em, column sep=3em]
{\hocolim \Tot(\tau_{\geq 2w}(X_i\wedge MU^{\bullet+1})) & \Tot(\tau_{\geq 2w}(X\wedge MU^{\bullet+1}))\\
 \hocolim \Tot^n(\tau_{\geq 2w}(X_i\wedge MU^{\bullet+1})) & \Tot^n(\tau_{\geq 2w}(X_i\wedge MU^{\bullet+1}))\\};

\path[thick, -stealth, font=\small]

(m-1-1) edge (m-1-2)
(m-1-2) edge (m-2-2)
(m-2-1) edge (m-2-2)
(m-1-1) edge (m-2-1);
\end{tikzpicture}
\end{center}
the bottom map is an equivalence for any $n$. Since $X_i$ and $X$ are bounded below, the vertical maps are isomorphisms on homotopy groups through a range increasing with $n$. Since the upper horizontal map is independent of $n$, it follows that it induces an isomorphism on all homotopy groups, and thus is an equivalence.
\end{proof}

\begin{prop}
\label{prop:Gamma-smash}
Let $X$ be a bounded below, finite type, even-cell complex, and let $Y$ have $MU$-homology concentrated in even degrees.
Then
\[
\Gamma_\star X \gammasmash \Gamma_\star Y
\lto \Gamma_\star (X \wedge Y)
\]
is an equivalence.
\end{prop}

\begin{proof}
Let $X$ be the homotopy colimit of 
\[
X^{(0)} \lto X^{(1)} \lto X^{(2)} \lto \cdots,
\]
with cofiber sequences
\[
X^{(n-1)} \lto X^{(n)} \lto S^{2k_n}.
\]
Corollary \ref{cor:Gamma-exact} shows that
\[
\Gamma_\star X^{(n-1)} \lto \Gamma_\star X^{(n)} \lto 
\Gamma_\star S^{2k_n}
\]
is a cofiber sequence.
As in the proof of Lemma \ref{lem:Gamma-seq-colimit},
\[
\Gamma_\star (X^{(n-1)} \wedge Y) \lto 
\Gamma_\star (X^{(n)} \wedge Y) \lto
\Gamma_\star (S^{2k_n} \wedge Y)
\]
is also a cofiber sequence.

We have a diagram
\begin{center}
\begin{tikzpicture}
\matrix (m) [matrix of math nodes, row sep=2em, column sep=3em]
{  
\Gamma_\star X^{(n-1)} \wedge_{\Gamma_\star S^0} \Gamma_\star Y &
\Gamma_\star X^{(n)} \wedge_{\Gamma_\star S^0} \Gamma_\star Y &
\Gamma_\star S^{2k_n} \wedge_{\Gamma_\star S^0} \Gamma_\star Y \\
\Gamma_\star (X^{(n-1)} \wedge Y) & \Gamma_\star (X^{(n)} \wedge Y) &
\Gamma_\star (S^{2k_n} \wedge Y) \\};

\path[thick, -stealth, font=\small]

(m-1-1) edge (m-1-2)
(m-1-2) edge (m-1-3)

(m-2-1) edge (m-2-2)
(m-2-2) edge (m-2-3)

(m-1-1) edge (m-2-1)
(m-1-2) edge (m-2-2)
(m-1-3) edge (m-2-3);
\end{tikzpicture}
\end{center}
in which the rows are cofiber sequences.
The right vertical map is an equivalence by 
Lemma \ref{lem:Gamma-even-suspend} and Example \ref{ex:Gamma-S^2}.
The left vertical map is an equivalence by an induction assumption.
Therefore, the middle vertical map is also an equivalence.

Apply homotopy colimits to obtain the equivalence
\[
\hocolim_n \left( \Gamma_\star X^{(n)} \gammasmash
\Gamma_\star Y \right)
\lto
\hocolim_n \left( \Gamma_\star (X^{(n)} \wedge Y) \right).
\]
The source of this map is equivalent to
$\Gamma_\star X \gammasmash \Gamma_\star Y$
because homotopy colimits commute with smash products,
and the target is equivalent
to $\Gamma_\star \left( X \wedge Y \right)$
by Lemma \ref{lem:Gamma-seq-colimit}.
\end{proof}

\section{The Steenrod algebra}
\label{sctn:Steenrod-algebra}

Consider the $\Gamma_\star S^0$-module $\Gamma_\star H\F_2$.
The Adams-Novikov spectral sequence for $H\F_2$
collapses.  It follows by inspection of the definition of
$\Gamma_\star$ that
$\Gamma_\star H\F_2$ is equal to the filtered spectrum
\[
\cdots * \lto * \lto H\F_2 \lto H\F_2 \lto \cdots,
\]
where the values are $*$ in filtrations greater than zero,
and the values are $H\F_2$ in filtrations less than or equal
to zero.
In particular,
$\pi_{\ast, \star} \Gamma_\star H \F_2$ is isomorphic to
$\F_2[\tau]$, where $\tau$ has degree $(0,-1)$.
Not coincidentally, these homotopy groups are isomorphic
to the motivic stable homotopy groups of the
$\C$-motivic Eilenberg-Mac Lane spectrum.

The goal of this section is to compute the Hopf algebra 
of co-operations on $\Gast H\F_2$.  
This Hopf algebra is the dual Steenrod algebra in the context
of $\Gast S^0$-modules.

\begin{defn}
Let $A_{*,\star}$ be the Hopf algebra
$\pi_{*, \star} \left( \Gamma_\star H\F_2 \gammasmash 
\Gamma_\star H\F_2 \right)_{*, \star}$.
\end{defn}

We begin by studying $\Gast BP$ and related objects.

\begin{example}
\label{ex:Gamma-BP}
The Adams-Novikov spectral sequence for $BP$
collapses.  It follows from the definition that
$\Gamma_\star BP$ is equal to the filtered spectrum
\[
\cdots \lto \tau_{\geq 4} BP \lto \tau_{\geq 2} BP \lto
BP \lto BP \lto \cdots.
\]

In particular, note that
$\pi_{\ast, \star} \Gamma_\star BP$ is isomorphic to
$\Z[\tau, v_1, v_2, \ldots]$, where $\tau$ has degree $(0,-1)$
and $v_i$ has degree $(2^{i+1}-2, 2^i - 1)$.
Not coincidentally, these homotopy groups are isomorphic
to the motivic stable homotopy groups of the
$\C$-motivic Brown-Peterson spectrum $BPGL$.
\end{example}

\begin{example}
\label{ex:Gamma-BP<n>}
Generalizing both $\Gast H\F_2$ and 
Example \ref{ex:Gamma-BP},
we obtain that 
the filtered spectrum $\Gamma_\star BP \langle n \rangle$ is 
\[
\cdots \lto \tau_{\geq 4} BP \langle n \rangle \lto 
\tau_{\geq 2} BP \langle n \rangle \lto
\tau_{\geq 0} BP \langle n \rangle \lto 
\tau_{\geq -2} BP \langle n \rangle \lto \cdots.
\]
\end{example}

If $R$ is a ring spectrum 
and $x$ is an indeterminant of degree $n$,
then we write $R[x]$ for the ring spectrum
$\bigvee_i \Sigma^{ni} R$,
with the obvious multiplication corresponding to multiplication
in a polynomial ring.
We use the same notation for a filtered ring spectrum $R$
and a bigraded indeterminant.
The object $R[x_0, x_1, \ldots]$ with multiple indeterminants
is defined analogously.

Recall that $BP \wedge BP$ is equivalent to
$BP[t_1, t_2, \ldots]$, where $t_i$ has degree
$2^{i+1} - 2$ \cite[Theorem 4.1.18]{Ravenel}.
We now establish an analogous result for $\Gamma_\star BP$.

\begin{prop}
\label{prop:BP-smash-BP}
The filtered spectrum
$\Gamma_\star BP \gammasmash \Gamma_\star BP$
is equivalent to the filtered spectrum
$\Gamma_\star BP [t_1, t_2, \ldots]$,
where $t_i$ has bidegree $(2^{i+1}-2, 2^i - 1)$.
\end{prop}

\begin{proof}
We observed in Example \ref{ex:even-cell} that
$BP$ is a bounded below, finite type, even-cell complex.
Therefore, Proposition \ref{prop:Gamma-smash} applies, and 
$\Gamma_\star BP \gammasmash \Gamma_\star BP$
is equivalent to
$\Gamma_\star (BP \wedge BP)$, which is equivalent to
$\Gamma_\star \left( BP [t_1, t_2, \ldots] \right)$.
An argument similar to the proof
of Lemma \ref{lem:Gamma-seq-colimit} shows that 
$\Gamma_\star$ commutes with the infinite wedge
that defines $BP[t_1, t_2, \ldots]$.
Finally, use Lemma \ref{lem:Gamma-even-suspend}
to determine the bidegree of $t_i$.
\end{proof}

\begin{prop}
\label{prop:BP_*BP}
The ring
$\pi_{*, \star} 
\left( \Gamma_\star BP \gammasmash \Gamma_\star BP \right)$
is isomorphic to 
\[
\Z[\tau][v_1, v_2, \ldots, t_1, t_2, \ldots],
\]
where $\tau$ has degree $(0,-1)$, and
$v_i$ and $t_i$ both have degree $(2^{i+1} - 2, 2^i - 1)$.
\end{prop}

\begin{proof}
This follows from Example \ref{ex:Gamma-BP} and
Proposition \ref{prop:BP-smash-BP}.
\end{proof}

\begin{defn}
The $\Gast H\F_2$-homology of a $\Gast S^0$-module $X$ is
\[
H_{*, \star} (X) = 
\pi_{*, \star} \left( X \gammasmash \Gast H\F_2 \right).
\]
\end{defn}

\begin{prop}
\label{prop:H_*BP}
The bigraded groups 
$H_{*, \star} \left( \Gamma_\star BP \right)$
are isomorphic to the free polynomial ring
$\F_2[\tau][\xi_1, \xi_2, \ldots ]$,
where $\tau$ has degree $(0,-1)$ and $\xi_n$ has degree
$(2^{n+1}-2, 2^n -1)$.
\end{prop}

\begin{proof}
Let $X(-1)$ be $\Gamma_\star BP$, and define
$X(n)$ inductively to be the cofiber of
\[
v_n: \Sigma^{2^{n+1} - 2, 2^n - 1} X(n-1) \lto X(n-1).
\]
From the descriptions of 
$\Gamma_\star H\F_2$ and $\Gamma_\star BP$
at the beginning of Section \ref{sctn:Steenrod-algebra},
we see that $\Gamma_\star H\F_2$ is $\hocolim_n X(n)$.
This mimics the standard construction of $H\F_2$
as $BP/(v_0, v_1, \ldots)$.

We have cofiber sequences
\[
\xymatrix@C=15pt{
\Sigma^{2^{n+1}-2, 2^n - 1}
\Gamma_\star BP \gammasmash X(n-1) 
\ar[r]^-{v_n} &
\Gamma_\star BP \gammasmash X(n-1) \ar[r] & 
\Gamma_\star BP \gammasmash X(n).
}
\]
Starting from Proposition \ref{prop:BP_*BP},
we can analyze the associated long exact sequences in homotopy groups.
Inductively, we compute that
$\pi_{*, \star} \left( \Gamma_\star BP \gammasmash X(n) \right)$ is 
isomorphic to
$\F_2[\tau][v_{n+1}, v_{n+2}, \ldots, \xi_1, \xi_2, \ldots ]$.
\end{proof}

\begin{remark}
The proof of Proposition \ref{prop:H_*BP}
follows an argument that could be used in the classical
case to the homology of $BP$ from prior knowledge of $BP_* BP$.
\end{remark}

\begin{prop}
\label{prop:H_*BP<n>}
The bigraded groups 
$H_{*, \star} \left( \Gamma_\star BP \langle n \rangle \right)$
are isomorphic to
\[
\frac{\F_2[\tau][\tau_{n+1}, \tau_{n+2}, \ldots, \xi_1, \xi_2, \ldots ]}
{\tau_i^2 + \tau \xi_{i+1}},
\]
where $\tau$ has degree $(0,-1)$,
$\tau_i$ has degree $(2^{i+1}-1, 2^i-1)$ and $\xi_i$ has degree
$(2^{i+1}-2, 2^i -1)$.
\end{prop}

\begin{proof}
Let $Y(n)$ be $\Gamma_\star BP$, and define
$Y(k)$ for $k> n$ inductively to be the cofiber of
\[
v_k: \Sigma^{2^{k+1} - 2, 2^k - 1} Y(k-1) \lto Y(k-1).
\]
From the descriptions of 
$\Gamma_\star BP$ and $\Gamma_\star BP \langle n \rangle$
at the beginning of Section \ref{sctn:Steenrod-algebra},
we see that $\Gamma_\star BP \langle n \rangle$ is 
$\hocolim_n Y(n)$.
This mimics the standard construction of $BP \langle n \rangle$
as $BP/(v_{n+1}, v_{n+2}, \ldots)$.

Suppose for induction that 
$H_{*, \star} \left( Y(k) \right)$
is isomorphic to
\[
\frac{\F_2[\tau][\tau_{n+1}, \tau_{n+2}, \ldots \tau_{k}, 
\xi_1, \xi_2, \ldots]}
{\tau_i^2 + \tau \xi_{i+1}}
\]
Proposition \ref{prop:H_*BP} establishes the base case.

Consider the cofiber sequence
\[
\xymatrix@C=14pt{
Y(k) \gammasmash \Gast H\F_2 \ar[r] & 
Y(k+1) \gammasmash \Gast H\F_2 \ar[r] &
\Sigma^{2^{k+2}-1, 2^{k+1} - 1} Y(k) \gammasmash \Gast H\F_2.
}
\]
Multiplication by $v_{k+1}$ is zero on 
$H_{*, \star} \left( Y(k) \right)$,
by Proposition \ref{prop:colim-Gamma} and the analogous 
classical fact.
Therefore, we have a short exact sequence
\[
\xymatrix@1{
H_{*, \star} \left( Y(k) \right) \ar[r] &
H_{*, \star} \left( Y(k+1) \right) \ar[r] &
H_{*, \star} \left( \Sigma^{2^{k+2}-1,2^{k+1}-1} Y(k) \right).
}
\]
This establishes the additive structure of 
$H_{*,\star} \left( Y(k+1) \right)$.  
\end{proof}

\begin{thm} \label{thm:steenrod}
The dual Steenrod algebra $A_{*, \star}$ is isomorphic to
\[
\frac{\F_2[\tau][\tau_0, \tau_1, \ldots, \xi_1, \xi_2, \ldots]}
{\tau_i^2 + \tau \xi_{i+1}},
\]
where the comultiplication is given by the formulas
\begin{equation*}
\Delta(\tau_i) = \tau_i \otimes 1 + \sum_{k=0}^i \ \xi_{i-k}^{2^k} \otimes \tau_{k},
\end{equation*}
and
\begin{equation*}
\Delta(\xi_i) = \sum_{k=0}^i \ \xi_{i-k}^{2^k} \otimes \xi_{k}.
\end{equation*} 
\end{thm}

\begin{proof}
The additive structure is given by the $n = 0$ case of
Proposition \ref{prop:H_*BP<n>}.

The formulas for the multiplication and comultiplication follow by 
Proposition \ref{prop:colim-Gamma} and the analogous
classical formulas.
\end{proof}

\section{Motivic modular forms}
\label{sctn:mmf}

In this section, we study the $\Gamma_\star S^0$-module
$\Gamma_\star \tmf$.  We will show that this
object has the desired properties of a ``motivic modular
forms" spectrum.
Note that $\Gamma_\star \tmf$ is an $E_\infty$-ring by
Proposition \ref{prop:Gamma-lax-monoidal}.

In Section \ref{sctn:Steenrod-algebra}, 
we worked with $\Gamma_\star H\F_2$-homology 
because the dual Steenrod algebra is easier to describe
than the Steenrod algebra.
We now work with the dual $\Gamma_\star H\F_2$-cohomology
because it is easier to state the specific computational
results that we are pursuing.

\begin{defn}
\label{defn:HF2-cohlgy}
The
$\Gamma_\star H\F_2$-cohomology of a 
$\Gast S^0$-module $X$ is
\[
H^{*,*} (X) = \pi_{*,*} F_{\Gast S^0} \left( X, \Gast H\F_2 \right),
\]
where $F_{\Gast S^0}(-,-)$ is the internal function object
in the category of $\Gast S^0$-modules.
\end{defn}

Under suitable bounded below, finite type assumptions on $X$,
the $\Gamma_\star H\F_2$-co\-hom\-ology of $X$ 
and the $\Gamma_\star H\F_2$-homology of $X$ are 
algebraic $\F_2[\tau]$-duals.

The main goal is to show that
the $\Gamma_\star H\F_2$-cohomology of
$\Gast \tmf$ is isomorphic to
$A\sslash A(2)$.
This implies that the cohomology of $A(2)$
is the $E_2$-page of the 
$\Gamma_\star H\F_2$-based 
Adams spectral sequence for $\Gamma_\star \tmf$.

\begin{defn}
\label{defn:A}
Let $A$ be the dual of $A_{*, \star}$.
Using the monomial basis of $A_{*, \star}$, let:
\begin{enumerate}
\item
$\Sq^1$ be dual to $\tau_0$.
\item
$Q_i$ be dual to $\tau_i$.
\item
$\Sq^{2^n}$ be dual to $\xi_1^{2^{n-1}}$ for $n \geq 1$.
\item
$P_j^0$ be dual to $\tau_{j-1}$ for $j \geq 1$.
\item
$P_j^i$ be dual to $\xi_j^{2^{i-1}}$ for $i \geq 1$ and $j \geq 1$.
\end{enumerate}
\end{defn}

We will need to refer to some quotients of the dual
Steenrod algebra $A_{*, \star}$ computed in
Theorem \ref{thm:steenrod}.

\begin{defn}
\label{defn:A(n)-E(n)}
\mbox{}
\begin{enumerate}
\item
Let $A(n)_{*, \star}$ be the quotient 
\[
\frac{\F_2[\tau][\tau_0, \tau_1, \ldots, \tau_n, 
\xi_1, \xi_2, \ldots, \xi_n]}
{\tau_i^2 + \tau \xi_{i+1}, 
\xi_1^{2^n}, \xi_2^{2^{n-1}}, \ldots, 
\xi_{n}^2, \tau_n^2}
\]
of $A_{*, \star}$, and let $A(n)$ be the dual subalgebra
of $A$.
\item
Let $E(n)_{*, \star}$ be the quotient
\[
\frac{\F_2[\tau][\tau_0, \tau_1, \ldots, \tau_n]}
{\tau_0^2, \tau_1^2, \ldots \tau_n^2}
\]
of $A_{*, \star}$, and let $E(n)$ be the dual subalgebra of $A$.
\end{enumerate}
\end{defn}

It is straightforward to check that
$E(n)$ is an exterior algebra on the elements $Q_0, Q_1, \ldots Q_n$,
and $A(n)$ is the subalgebra of $A$ generated by
$\Sq^1$, $\Sq^2$, \ldots, $\Sq^{2^n}$.

\begin{defn}
We define the finite even-cell complexes:
\begin{enumerate}
\item
$X$ is the cofiber of $\nu: S^3 \lto S^0$.
\item
$Y$ is the cofiber of $\eta: \Sigma X \lto X$.
\item
$Z$ is the cofiber of $w_1: \Sigma^5 Y \lto Y$.
\end{enumerate}
\end{defn}

The best way to describe $w_1$ is in terms of
the Adams spectral sequence for maps
$\Sigma^5 Y \lto Y$.  It is detected by the element
whose May spectral sequence name is $h_{21}$.
This is essential below in Lemma \ref{lemma:cofibrefinitecomplex}
when we relate $w_1$ to the Steenrod operation $P_2^1$.

For our purposes, the essential property of $Z$ is that
$\tmf \wedge Z$ is equivalent to $BP\langle 2 \rangle$
\cite{Mathew16}.
We will start with the cohomology of 
$\Gast \tmf \gammasmash \Gast Z$, and then work backwards to obtain
the cohomology of 
$\Gast \tmf \gammasmash \Gast Y$, 
$\Gast \tmf \gammasmash \Gast X$, and finally
$\Gast \tmf$.
The basic idea is not original.  See, for example,
\cite[Proposition 1.7]{Wilson75}
for an analogous argument that computes the
homology of $BP\langle n \rangle$.  
See also \cite[Section 5]{IS} for a motivic version
of an argument that computes the cohomology of $ko$.

\begin{lemma}
\label{lemma:cofibrefinitecomplex}
In the category of $\Gamma_\star S^0$-modules,
there are cofiber sequences
\[
\xymatrix@1{
\Sigma^{3,2} \Gast S^0 \ar[r]^-{\Gast \Sigma \nu} &
\Gast S^0 \ar[r]^-{\iota_X} & 
\Gast X \ar[r]^-{\pi_X} &
\Sigma^{4,2} \Gast S^0,
}
\]
\[
\xymatrix@1{
\Sigma^{1,1} \Gast X \ar[r]^-{\Gast \Sigma \eta} &
\Gast X \ar[r]^-{\iota_Y} & 
\Gast Y \ar[r]^-{\pi_Y} & \Sigma^{2,1} \Gast X,
}
\]
and
\[
\xymatrix@1{
\Sigma^{5,3} \Gast Y \ar[r]^-{\Gast \Sigma w_1} &
\Gast Y \ar[r]^-{\iota_Z} & 
\Gast Z \ar[r]^-{\pi_Z} &
\Sigma^{6,3} \Gast Y.
}
\]
In $\Gast H\F_2$-cohomology, the compositions
\[
\xymatrix@1{
\Gast X \ar[r]^-{\pi_X} & 
\Sigma^{4,2} \Gast S^0 \ar[r]^-{\Gast \nu} & 
\Sigma^{4,2} \Gast X,
}
\]
\[
\xymatrix@1{
\Gast Y \ar[r]^-{\pi_Y} & 
\Sigma^{2,1} \Gast X \ar[r]^-{\Gast \eta} & 
\Sigma^{1,0} \Gast X,
}
\]
and
\[
\xymatrix@1{
\Gast Z \ar[r]^-{\pi_Z} & 
\Sigma^{6,3} \Gast Y \ar[r]^-{\Gast w_1} & 
\Sigma^{6,3} \Gast Z
}
\]
are multiplication by $\Sq^4 = (\xi_1^2)^\vee$,
$\Sq^2 = \xi_1^\vee$,
and $P_2^1 = \xi_2^\vee$ respectively.
\end{lemma}

\begin{proof}
Apply Corollary \ref{cor:Gamma-exact}
to the cofiber sequence
\[
\xymatrix@1{
S^0 \ar[r] & X \ar[r] & S^4
}
\]
to obtain the cofiber sequence
\[
\xymatrix@1{
\Gast S^0 \ar[r] & \Gast X \ar[r] & \Gast S^4.
}
\]
Use Lemma \ref{lem:Gamma-even-suspend} to identify the third term,
and then rotate to obtain the cofiber sequence in part (1).
The arguments for part (2) and part (3) are essentially identical.

The assertion about the action in cohomology follows by 
Proposition \ref{prop:colim-Gamma} and the analogous
classical facts.
One way to understand the maps in classical cohomology
is to observe that $\nu$, $\eta$, and $w_1$ are represented
in the Adams spectral sequence in filtration $1$. They are
detected by $h_2$, $h_1$, and $h_{21}$, which have cobar
representatives $[\zeta_1^2]$, $[\zeta_1]$, and $[\zeta_2^2]$
respectively.
\end{proof}

\begin{lemma}
\label{lem:mutmf-even-degrees}
$MU_* \tmf$ is concentrated in even degrees. (Compare the stronger Corollary 5.2 in \cite{Mathew16}.)
\end{lemma}
\begin{proof}
Consider the maps
\begin{align*}
\nu: S^3 \lto S^0\\
\eta: \Sigma X \lto X\\
w_1: \Sigma^5 Y \lto Y.
\end{align*}
The targets of these maps are bounded-below, finite type, even cell
complexes, while the sources are suspensions of bounded-below,
finite type, even cell complexes.
For degree reasons, these maps are zero in $MU$-homology.
After smashing with $MU$, all three maps
have sources and targets that are free $MU$-modules
by Lemma \ref{lem:MU_*X-even}.
This means that they are null-homotopic after smashing
with $MU$, since they are zero
in $MU$-homology.

We get short exact sequences
\begin{gather*}
0 \lto MU_* \tmf \lto MU_*(\tmf \wedge X) \lto MU_*(\Sigma^4 \tmf)\lto 0\\
0 \lto MU_* (\tmf \wedge X) \lto MU_*(\tmf \wedge Y) \lto MU_*(\Sigma^2 \tmf \wedge X) \lto 0\\
0 \lto MU_* (\tmf \wedge Y) \lto MU_*(\tmf \wedge Z) \lto MU_*(\Sigma^6 \tmf \wedge Z) \lto 0
\end{gather*}
Starting with the classical equivalence $BP\langle 2 \rangle \simeq \tmf \wedge Z$ \cite{Mathew16}, we see that
$MU_*(\tmf \wedge Z)$ is concentrated in even degrees, since $MU_*(BP\langle 2\rangle)$ is. The exact sequences then imply that $MU_*(\tmf\wedge Y)$, $MU_*(\tmf \wedge X)$ and finally $MU_*(\tmf)$ are also concentrated in even degrees.
\end{proof}

\begin{prop}
\label{prop:H-tmf-smash-Z}
$H^{*,\star} \left( \Gamma_\star \tmf \gammasmash \Gamma_\star Z \right)$
is isomorphic to $A\sslash E(2)$.
\end{prop}

\begin{proof}
Starting with the equivalence
$BP\langle 2 \rangle \simeq \tmf \wedge Z$,
apply Proposition \ref{prop:Gamma-smash} together with Lemma \ref{lem:mutmf-even-degrees} to see that
\[
\Gamma_\star BP\langle 2 \rangle \simeq 
\Gamma_\star \tmf \gammasmash \Gamma_\star Z.
\]
Dualizing 
Proposition \ref{prop:H_*BP<n>}, we have that
$H^{*,\star} \left( \Gamma_\star BP\langle 2 \rangle \right)$
is isomorphic to $A\sslash E(2)$.
\end{proof}

\begin{prop}
\label{prop:H-tmf-smash-Y}
$H^{*,\star} \left( \Gamma_\star \tmf \gammasmash \Gamma_\star Y \right)$
is isomorphic to $A\sslash F$, where $F$ is the subalgebra of $A$ generated
by $Q_0$, $Q_1$, $Q_2$, and $P_2^1$.
\end{prop}

Note that the dual $F_{*,\star}$ of $F$ is
\[
\frac{\F_2[\tau][\tau_0, \tau_1, \tau_2, \xi_2 ]}
{\tau_0^2, \tau_1^2 = \tau \xi_2, \xi_2^2, \tau_2^2}.
\]

\begin{proof}
By the second part of Lemma \ref{lemma:cofibrefinitecomplex},
the composition
\[
A\sslash E(2)\lto H^{*,\star} \left( \Gamma_\star \tmf \gammasmash \Gamma_\star Z \right)  \lto 
H^{*,\star} \left( \Gamma_\star \tmf \gammasmash \Gamma_\star Y \right)
\]
factors through
$A\sslash F$.
Therefore, we have a diagram
\[
\xymatrix{
& 0 \\
H^{*,\star} \left( \Gamma_\star \tmf \gammasmash \Gamma_\star Y \right)
\ar[u] & 
A\sslash F \ar[l] \ar[u] \\
H^{*,\star} \left( \Gamma_\star \tmf \gammasmash \Gamma_\star Z \right)
\ar[u] &
A\sslash E(2) \ar[u] \ar[l]_-{\cong} \\
H^{*,\star} \Sigma^{6,3} 
\left( \Gamma_\star \tmf \gammasmash \Gamma_\star Y \right) 
\ar[u] &
\Sigma^{6,3} A\sslash F \ar[l] \ar[u] \\
\ar[u] & 0. \ar[u] 
}
\]
in which the middle horizontal arrow is an isomorphism
by Proposition \ref{prop:H-tmf-smash-Z}, 
and the top and bottom horizontal arrows
are the same.

A diagram chase shows that the bottom horizontal map (and therefore
the top as well) is an injection.  An inductive argument now
shows that those maps are surjections as well.
\end{proof}

\begin{prop}
\label{prop:H-tmf-smash-X}
$H^{*,\star} \left( \Gamma_\star \tmf \gammasmash \Gamma_\star X \right)$
is isomorphic to $A\sslash G$, where $G$ is the subalgebra of $A$ generated
by $Q_0$, $Q_2$, $P_2^1$, and $\Sq^2$.
\end{prop}

Note that the dual $G_{*,\star}$ of $G$ is
\[
\frac{\F_2[\tau][\tau_0, \tau_1, \tau_2, \xi_1, \xi_2 ]}
{\tau_0^2 = \tau \xi_1, \xi_1^2, \tau_1^2 = \tau \xi_2, 
\xi_2^2, \tau_2^2}.
\]

\begin{proof}
The proof is essentially the same as the proof of 
Proposition \ref{prop:H-tmf-smash-Y}, using the diagram
\[
\xymatrix{
& 0 \\
H^{*,\star} \left( \Gamma_\star \tmf \gammasmash \Gamma_\star X \right)
\ar[u] &
 A\sslash G \ar[l] \ar[u] \\
H^{*,\star} \left( \Gamma_\star \tmf \gammasmash \Gamma_\star Y \right)
\ar[u] &
A\sslash F \ar[u] \ar[l]_-{\cong} \\
H^{*,\star} \Sigma^{2,1} 
\left( \Gamma_\star \tmf \gammasmash \Gamma_\star Y \right) 
\ar[u] &
\Sigma^{2,1} A\sslash G \ar[l] \ar[u] \\
 \ar[u] & 0. \ar[u]
}
\]
\end{proof} 

\begin{thm} \label{thm:homologymmf}
$H^{*,\star} \left( \Gamma_\star \tmf \right)$
is isomorphic to $A\sslash A(2)$. 
\end{thm}

Equivalently,
the $\Gast H\F_2$-homology of $\Gast \tmf$ is
$A_{*,\star} \square_{A(2)_{*,\star}} \F[\tau]$.

\begin{proof}
The proof is essentially the same as the proofs of Propositions
\ref{prop:H-tmf-smash-Y} and \ref{prop:H-tmf-smash-X}, 
using the diagram
\[
\xymatrix{
& 0 \\
H^{*,\star} \left( \Gamma_\star \tmf \right) \ar[u] & 
A\sslash A(2) \ar[u] \ar[l] \\
H^{*,\star} \left( \Gamma_\star \tmf \gammasmash \Gamma_\star X \right)
\ar[u] &
A\sslash G \ar[u] \ar[l]_-{\cong} \\
H^{*,\star} \Sigma^{4,2} \left( \Gamma_\star \tmf \right) \ar[u] & 
\Sigma^{4,2} A\sslash A(2) \ar[u] \ar[l] \\
 \ar[u] & 0. \ar[u] \\
}
\]
\end{proof}

\begin{remark}
The same technique can be used to compute that
$H^{*, \star} \left( \Gast ko \right)$ equals $A\sslash A(1)$.
In this case, one uses the cofiber sequences
\[
\xymatrix@1{
H\Z \ar[r]^{2} & H\Z \ar[r] & H\F_2,
}
\]
\[
\xymatrix@1{
\Sigma^2 ku \ar[r]^{v_1} & ku \ar[r] & H\Z,
}
\]
and
\[
\xymatrix@1{
\Sigma^1 ko \ar[r]^\eta & ko \ar[r] & ku.
}
\]
See \cite[Section 5]{IS} or \cite[Proposition 1.7]{Wilson75} 
for similar arguments.
\end{remark}

\section{Comparison to $\C$-motivic homotopy theory}
\label{sctn:compare}

The goal of this section is to show that the $\infty$-category
of $\Gast S^0$-modules is equivalent to the cellular 2-complete
$\C$-motivic stable $\infty$-category $\Spt_\C$.
Equivalences in $\Spt_\C$ are detected by the bigraded stable
motivic homotopy groups, and every object is equivalent
to an object that can be built from spheres by homotopy colimits.

Recall that $\C$-motivic stable homotopy theory is enriched over
spectra in the following sense.  For any motivic spectra $X$ and
$Y$, there exists a mapping spectrum $F_s(X,Y)$.
If $X$ and $Y$ are homotopically well-behaved, then
$\pi_k F_s(X,Y)$ equals $[ \Sigma^{k,0} X, Y ]$.
The subscript $s$ indicates that we are considering only the
function object as a classical spectrum, not as a motivic spectrum.

Motivic spectra are also tensored over spectra, in the sense that
there is a motivic spectrum $K \otimes X$ for every spectrum $K$
and motivic spectrum $X$.

We will rely on the Betti realization functor
$B: \Spt_\C \to \Spt$ from the category of cellular 2-complete
$\C$-motivic spectra to the category of ordinary 2-complete spectra.
Recall that $B(S^{p,q})$ is $S^p$.  Also, the map
$\tau: S^{0,-1} \to S^{0,0}$ realizes to the identity
$S^0 \to S^0$.

\begin{lemma}
\label{lem:Betti-equiv}
If $p \leq q$, then 
Betti realization induces an equivalence
\[
\Map_{\Spt_\C}(S^{0,p}, S^{0,q}) \xrightarrow{\simeq} 
\Map_{\Spt}(S^0,S^0).
\]
\end{lemma}

\begin{proof}
This corresponds to the observation in \cite{GI}
that Betti realization induces an isomorphism
$\pi_{s,w} \to \pi_s$ when $w \leq 0$, which in turn follows from
naive considerations of the $\C$-motivic Adams-Novikov spectral sequence.
\end{proof}

The key point of Lemma \ref{lem:Betti-equiv}
is that when $p \leq q$, a map $S^{0,p} \to S^{0,q}$
is uniquely determined up to homotopy by its Betti realization.

Now consider $\Z$ as a poset category with respect to $\leq$,
with symmetric-monoidal structure obtained from addition. 

\begin{lemma}
\label{lem:monoidal-S0,*}
Up to contractible choice,
there is a unique symmetric-monoidal functor 
$S^{0,\star}: N(\Z)\to \Spt_\C$ that sends $n \mapsto S^{0,n}$ 
such that $B(S^{0, \star})$ is the constant functor
$N(\Z)\to \Spt$ with value $S^0$.
The induced maps $S^{0,n} \to S^{0,n+1}$ are homotopic to $\tau$.
\end{lemma}

\begin{proof}
A lax symmetric-monoidal functor between symmetric-monoidal 
$\infty$-cate\-gories $\calC\to \calD$ is the same as a functor 
$\calC^{\otimes} \to \calD^{\otimes}$
of the associated colored operads.
Thus the data of a lax symmetric-monoidal functor $F: N(\Z) \to \calC$ gives, for each tuple $(i_1,\ldots, i_k; n)$ with $i_1 + \ldots + i_k \leq n$, a point in the associated multi-mapping space $\Map_{\calC}(F(i_1)\otimes \ldots \otimes F(i_k), F(n))$, with coherence homotopies between the various ways to compose and permute the domain factors.

For our desired functor $S^{0,\star}$, the relevant multi-mapping spaces in $\Spt_\C$ are of the form $\Map_{\Spt_\C}(S^{0,i_1+\ldots + i_k}, S^{0,n})$. 
Betti realization induces an equivalence
\begin{equation}
\label{multimappingequiv}
\Map_{\Spt_\C}(S^{0,i_1+\cdots + i_k}, S^{0,n}) \xrightarrow{\simeq} \Map_{\Spt}(S^0,S^0)
\end{equation}
by Lemma \ref{lem:Betti-equiv}.

The constant functor $N(\Z) \to \Spt$ with value $S^0$ 
is lax symmetric-monoidal, since $S^0$ is an $E_\infty$-ring spectrum. 
The corresponding coherence data on multi-mapping spaces lifts uniquely (up to contractible choice) along the equivalences \eqref{multimappingequiv}, giving us the desired lax symmetric-monoidal functor $S^{0,\star}$. To see that it is symmetric-monoidal, 
it is enough to observe that the lax structure maps
\[
S^{0,i_1} \otimes \cdots \otimes S^{0,i_k}\to S^{0,i_1+\ldots + i_k}
\]
are in fact equivalences 
by Lemma \ref{lem:Betti-equiv}
because their Betti realizations are the equivalences
$S^0\otimes\cdots\otimes S^0 \simeq S^0$.

Finally, we must identify the maps $S^{0,n} \to S^{0,n+1}$.
The Betti realizations of these maps are the identity on $S^0$,
so they must be homotopic to $\tau$ by Lemma \ref{lem:Betti-equiv}.
\end{proof}

\begin{defn}
\label{defn:Omega}
Define $\Omega^{0,\star}_s: \Spt_\C \to \SptZ$ to be the 
filtered spectrum
\[
\cdots \to F_s(S^{0,n+1}, X) \to F_s(S^{0,n}, X) \to
F_s(S^{0,n-1}, X) \to \cdots,
\]
where the structure maps are determined by the maps
of $S^{0,\star}$ from Lemma \ref{lem:monoidal-S0,*}.
\end{defn}

The notation reflects the fact that 
the definition of $\Omega^{0,\star}_s$ is similar to the usual
desuspension functor, except that we use the spectrum-valued
function object instead of the internal function object.
Another possible description of $\Omega_s^{0,\star}$
is that it corresponds under the equivalence between functors $\Spt_\C \to \Spt^{\Z^{op}}$ and functors $\Spt_\C \times N(\Z)^{op}\to \Spt$ to
\[
\Spt_\C \times N(\Z)^{op} \to \Spt:
(X, n) \mapsto F_s(S^{0,n}, X).
\]
We will use this adjoint description below in 
Proposition \ref{prop:Omega-monoidal}.

\begin{remark}
\label{rem:htpy-gp-Omega}
Recall from Section \ref{sec:filteredspectra}
that the bigraded homotopy group $\pi_{i,j} (Y_\star)$
of a filtered spectrum $Y_\star$ is equal to
$\pi_i Y_j$.
For any motivic spectrum $X$, we then have that
$\pi_{i,j} ( \Omega^{0,\star}_s X)$
is equal to $\pi_i F_s(S^{0,j}, X)$, which equals
the motivic stable homotopy group $\pi_{i,j} X$.
This observation is precisely the point of the construction
of $\Omega^{0,\star}_s$; it records the motivic stable homotopy
groups of $X$ in a filtered spectrum.
\end{remark}

\begin{defn}
\label{defn:otimes-S0,*}
Define $- \otimes S^{0,\star}: \SptZ \to \Spt_\C$ to be the left
adjoint of $\Omega_s^{0,\star}$.
\end{defn}

More concretely, for any filtered spectrum $Y_\star$,
the motivic spectrum
$Y_\star \otimes S^{0,\star}$ equals
\[
\hocolim_{i \geq j} Y_i \otimes S^{0,j}.
\]

\begin{prop}
\label{prop:Omega-monoidal}
The functor $\Omega^{0,\star}_s$
is lax symmetric-monoidal with respect to Day convolution.
\end{prop}

\begin{proof}
By formal properties of Day convolution $\Spt_\C \to \Spt^{N(\Z)^{op}}$ is lax symmetric monoidal precisely when the corresponding functor $\Spt_\C \times N(\Z)^{op} \to \Spt$ is. This can be written as the composition
\[
\Spt_\C \times N(\Z)^{op} \xrightarrow{\id \times S^{0,\star}} \Spt_\C\times \Spt_\C^{op} \xrightarrow{F_s} \Spt.
\]
The first functor is symmetric-monoidal by 
Lemma \ref{lem:monoidal-S0,*}, while the second functor
is lax symmetric-monoidal by the standard functoriality of
mapping spectra.
Therefore, the composition is lax symmetric-monoidal.
\end{proof}

One consequence of 
Proposition \ref{prop:Omega-monoidal}
is that $\Omega^{0,\star}_s X$ is an
$\Omega^{0,\star}_s(S^{0,0})$-module for all $X$.
Our next goal is to identify
$\Omega^{0,\star}_s(S^{0,0})$.

Betti realization induces a map
\[
F_s \left( S^{0,w}, MGL^{n+1} \right) \map
F(S^0, MU^{n+1}) = MU^{n+1},
\]
where the object on the right is the usual function object
for classical spectra.
Lemma \ref{lem:MGL-MU-compare} determines this map more
explicitly.

\begin{lemma}
\label{lem:MGL-MU-compare}
For all $w$ and $n$, 
the map $F_s(S^{0,n}, MGL^{n+1}) \map MU^{n+1}$
is the canonical map
\[
\tau_{\geq 2w} \left( MU^{n+1} \right) \map MU^{n+1}.
\]
\end{lemma}

\begin{proof}
Recall that $\pi_* MU$ is isomorphic to the Lazard ring
$\Z [ x_1, x_2, \ldots]$, where $x_i$ has degree $2i$ 
\cite{Milnor60}.
Also, $\pi_{*,*} MGL$ is isomorphic to $\Z[\tau][x_1, x_2, \ldots]$,
where $x_i$ has degree $(2i,i)$ \cite[Theorem 7]{HKO11}.   
Moreover, Betti realization takes $x_i$ to $x_i$.

The classical homotopy groups of $F_s \left( S^{0,w}, MGL \right)$
are equal to the motivic homotopy groups
$\pi_{*,w} \left( MGL \right)$.  From the description in
the previous paragraph, we see that
these groups are the same as the homotopy groups
of $\tau_{\geq 2w} \left( MU \right)$, and that
Betti realization
\[
F_s(S^{0,n}, MGL) \map MU
\]
induces an isomorphism on homotopy groups in degrees $2w$ and above.
This establishes the result for the case $n = 0$.

The proof for $n > 0$ is similar, using that
$MU^{n+1}$ splits as a wedge of even suspensions 
of $MU$ \cite[p.\ 87]{Adams74}, and
$MGL^{n+1}$ splits correspondingly as a wedge of
$\Sigma^{2k,k}$-suspensions of $MGL$ \cite{NSO15}.
\end{proof}

\begin{prop}
\label{prop:omega-gamma}
The filtered spectrum
$\Omega^{0,\star}_s S^{0,0}$ is equivalent
to $\Gast S^0$.
\end{prop}

\begin{proof}
Using that the motivic Adams-Novikov spectral sequence converges,
we have that $S^{0,0}$ is equivalent to 
$\Tot \left( MGL^{\bullet + 1} \right)$,
where $MGL^{\bullet + 1}$ is the standard
cosimplicial $MGL$-resolution of $S^{0,0}$.
Therefore, for each $w$, the spectrum
$\Omega^{0,w}_s S^{0,0}$ is equivalent to
$\Omega^{0,w}_s \Tot \left( MGL^{\bullet + 1} \right)$.
Using that homotopy limits commute with function spectra,
this object is equivalent to
$\Tot \left( F_s (S^{0,w}, MGL^{\bullet + 1}) \right)$.
Lemma \ref{lem:MGL-MU-compare}
identifies this spectrum with
$\Tot \left( \tau_{\geq 2w} MU^{\bullet + 1} \right)$,
which is precisely the definition of
$\Gamma_w S^0$.
\end{proof}

Propositions \ref{prop:Omega-monoidal} and \ref{prop:omega-gamma}
show that the functor 
$\Omega^{0, \star}_s$ takes values
in $\Gast S^0$-modules, rather than just filtered spectra.
From this perspective, we can define its left adjoint
as a functor from $\Gast S^0$-modules to 
$\C$-motivic spectra.

\begin{defn}
Let
\[
\Mod_{\Gast S^0} \to \Spt_\C:
Y_\star \mapsto Y_\star \gammaotimes S^{0,\star}
\]
be the left adjoint to the functor
$\Omega^{0,\star}_s: \Spt_\C \to \Mod_{\Gast S^0}$.
\end{defn}

We can describe $Y_\star \gammaotimes S^{0,\star}$ in more
concrete terms, although this description is not strictly
necessary.  It is the geometric realization of the simplicial bar construction
\[
\xymatrix{
\cdots\ar@<-.8ex>[r] \ar@<.8ex>[r] \ar[r] &
Y_\star \otimes \Gast S^0 \otimes S^{0,\star} 
\ar@<-.5ex>[r] \ar@<.5ex>[r] &
Y_\star \otimes S^{0,*}.
}
\]
In each degree, all but the last tensor symbol indicate Day convolution
of two filtered spectra, while the 
last tensor symbol represents the functor of
Definition \ref{defn:otimes-S0,*}.

Recall the suspension functor $\Sigma^{i,j}$ for
filtered spectra described in Section \ref{sec:filteredspectra}.

\begin{lemma}
\label{lem:gammaotimes-sphere}
The motivic spectrum
$(\Sigma^{i,j} \Gast S^0) \gammaotimes S^{0,\star}$ is
equivalent to $S^{i,j}$.
\end{lemma}

\begin{proof}
By adjointness, maps
$(\Sigma^{i,j} \Gast S^0) \gammaotimes S^{0,\star} \map Y$
correspond to maps
$S^{i,j} \to \Omega^{0,\star}_s Y$ of filtered spectra.
From the definition of $S^{i,j}$ in
Definition \ref{defn:SptZ-spheres},
such maps correspond to maps
$S^i \to F_s(S^{0,j}, Y)$ in spectra, 
which in turn correspond to maps
$S^{i,j} \to Y$ in motivic spectra.
\end{proof}

\begin{prop}
\label{prop:htpy-gp-gammaotimes}
For any $\Gast S^0$-module $Y_\star$, the functor $- \gammaotimes S^{0,\star}$ induces an isomorphism
\[
\pi_{i,j}(Y_\star) = [\Sigma^{i,j} \Gamma_\star S^0, Y_\star] \to [S^{i,j}, Y_\star\gammaotimes S^{0,\star}] = \pi_{i,j}(Y_\star\gammaotimes S^{0,\star}).
\]
\end{prop}

\begin{proof}
First suppose that $Y_\star$ is of the form
$\Sigma^{p,q} \Gast S^0$.
Then
$Y_\star \gammaotimes S^{0,\star}$ is equal to
$S^{p,q}$ by Lemma \ref{lem:gammaotimes-sphere}, so
$\pi_{i,j} (Y_\star \gammaotimes S^{0,\star})$ equals
the motivic stable homotopy group $\pi_{i-p,j-q}$.
On the other hand, $\pi_{i,j} (Y_\star)$ is equal to
$\pi_{i-p}( \Gamma_{j-q} S^0)$, which
equals 
$\pi_{i-p} F_s(S^{0,j-q}, S^{0,0})$
by Proposition \ref{prop:omega-gamma}.
Finally, this last group also equals the 
motivic stable homotopy group $\pi_{i-p,j-q}$.
Thus, the proposition holds when $Y_\star$ is of the form
$\Sigma^{p,q} \Gast S^0$.

Let $\mathcal{C}$ be the class of all $\Gast S^0$-modules
$Y_\star$ for which the proposition is true.
We have just shown that $\mathcal{C}$ contains
$\Sigma^{p,q} \Gast S^0$.
Moreover, $\mathcal{C}$ is closed under arbitrary coproducts and filtered colimits
by compactness of the classical sphere $S^i$ and of 
the motivic sphere $S^{i,j}$.
Finally, if
\[
X_\star \to Y_\star \to Z_\star
\]
is a cofiber sequence of $\Gast S^0$-modules and any two
of $X_\star$, $Y_\star$, and $Z_\star$ belong to
$\mathcal{C}$, then the third belongs to $\mathcal{C}$ as well
by the five lemma applied to long exact sequences of homotopy
groups.

Finally, Proposition \ref{prop:Gamma-generate} implies that
$\mathcal{C}$ equals the entire category of
$\Gast S^0$-modules.
\end{proof}

\begin{thm}
The functors
$- \gammaotimes S^{0,\star}$ and
$\Omega^{0,\star}_s$ are inverse equivalences between
$\Spt_\C$ and $\Mod_{\Gast S^0}$.
\end{thm}

\begin{proof}
Since equivalences in $\Spt_\C$ and in $\Mod_{\Gast S^0}$
are both detected by bigraded homotopy groups, it suffices to
show that the counit map
$(\Omega^{0, \star}_s X) \gammaotimes S^{0,\star} \to X$
induces an isomorphism of motivic stable homotopy groups, and that
the unit map
$Y_\star \to \Omega^{0,\star} ( Y_\star \gammaotimes S^{0,\star} )$
induces an isomorphism of bigraded homotopy groups of
filtered spectra.
These claims follow from
Remark \ref{rem:htpy-gp-Omega} and Proposition \ref{prop:htpy-gp-gammaotimes}.
\end{proof}

\bibliographystyle{amsalpha}
\bibliography{Cmmf}

\providecommand{\bysame}{\leavevmode\hbox to3em{\hrulefill}\thinspace}
\providecommand{\MR}{\relax\ifhmode\unskip\space\fi MR }
% \MRhref is called by the amsart/book/proc definition of \MR.
\providecommand{\MRhref}[2]{%
  \href{http://www.ams.org/mathscinet-getitem?mr=#1}{#2}
}
\providecommand{\href}[2]{#2}
\begin{thebibliography}{DFHH14}

\bibitem[Ada74]{Adams74}
J.~F. Adams, \emph{Stable homotopy and generalised homology}, University of
  Chicago Press, Chicago, Ill.-London, 1974, Chicago Lectures in Mathematics.
  \MR{0402720}

\bibitem[DFHH14]{tmfbook}
Christopher~L. Douglas, John Francis, Andr\'e~G. Henriques, and Michael~A. Hill
  (eds.), \emph{Topological modular forms}, Mathematical Surveys and
  Monographs, vol. 201, American Mathematical Society, Providence, RI, 2014.
  \MR{3223024}

\bibitem[Ghe]{GheCt}
Bogdan Gheorghe, \emph{The motivic cofiber of $\tau$}, Doc.\ Math., to appear.

\bibitem[Ghe17]{GheKwn}
\bysame, \emph{Exotic motivic periodicities}, 2017, preprint, arXiv:1709.00915.

\bibitem[GI16]{GI}
Bogdan Gheorghe and Daniel~C. Isaksen, \emph{The structure of motivic homotopy
  groups}, Bolet{\'i}n de la Sociedad Matem{\'a}tica Mexicana (2016), 1--9.

\bibitem[Gla16]{Glas16}
Saul Glasman, \emph{Day convolution for {$\infty$}-categories}, Math. Res.
  Lett. \textbf{23} (2016), no.~5, 1369--1385. \MR{3601070}

\bibitem[GWX18]{GWX}
Bogdan Gheorghe, Guozhen Wang, and Zhouli Xu, \emph{The special fiber of the
  motivic deformation of the stable homotopy category is algebraic}, 2018,
  preprint, arXiv:1809.09290.

\bibitem[Hei17]{Heine17}
Hadrian Heine, \emph{A characterization of cellular motivic spectra}, preprint,
  arXiv:\-1712.\-00521, 2017.

\bibitem[HKO11]{HKO11}
Po~Hu, Igor Kriz, and Kyle Ormsby, \emph{Remarks on motivic homotopy theory
  over algebraically closed fields}, J. K-Theory \textbf{7} (2011), no.~1,
  55--89. \MR{2774158 (2012b:14040)}

\bibitem[HPS97]{HPS}
Mark Hovey, John~H. Palmieri, and Neil~P. Strickland, \emph{Axiomatic stable
  homotopy theory}, Mem. Amer. Math. Soc. \textbf{128} (1997), no.~610, x+114.
  \MR{1388895}

\bibitem[IS11]{IS}
Daniel~C. Isaksen and Armira Shkembi, \emph{Motivic connective {$K$}-theories
  and the cohomology of {A}(1)}, J. K-Theory \textbf{7} (2011), no.~3,
  619--661. \MR{2811718}

\bibitem[Isa]{StableStems}
Daniel Isaksen, \emph{Stable stems}, Mem.\ Amer.\ Math.\ Soc., to appear.

\bibitem[Isa09]{Isaksen-A(2)}
Daniel~C. Isaksen, \emph{The cohomology of motivic {$A(2)$}}, Homology Homotopy
  Appl. \textbf{11} (2009), no.~2, 251--274. \MR{2591921}

\bibitem[IWX]{IWX18}
Daniel~C. Isaksen, Guozhen Wang, and Zhouli Xu, \emph{More stable stems}, in
  preparation.

\bibitem[Kra18]{Krause18}
Achim Krause, \emph{Periodicity in motivic homotopy theory and over ${BP}_*
  {BP}$}, Ph.D. thesis, Universitaet Bonn, 2018.

\bibitem[Lur17]{HA}
Jacob Lurie, \emph{Higher algebra}, \url{math.harvard.edu/~lurie}, 2017.

\bibitem[Mat16]{Mathew16}
Akhil Mathew, \emph{The homology of tmf}, Homology Homotopy Appl. \textbf{18}
  (2016), no.~2, 1--29. \MR{3515195}

\bibitem[Mil60]{Milnor60}
J.~Milnor, \emph{On the cobordism ring {$\Omega \sp{\ast} $} and a complex
  analogue. {I}}, Amer. J. Math. \textbf{82} (1960), 505--521. \MR{0119209}

\bibitem[NSO15]{NSO15}
Niko Naumann, Markus Spitzweck, and Paul~Arne \O{}stv\ae{}r, \emph{Existence
  and uniqueness of {$E_\infty$} structures on motivic {$K$}-theory spectra},
  J. Homotopy Relat. Struct. \textbf{10} (2015), no.~3, 333--346. \MR{3385689}

\bibitem[Pst18]{Pstragowski18}
Piotr Pstragowski, \emph{Synthetic spectra and the cellular motivic category},
  preprint, arXiv:1803.01804, 2018.

\bibitem[Rav86]{Ravenel}
Douglas~C. Ravenel, \emph{Complex cobordism and stable homotopy groups of
  spheres}, Pure and Applied Mathematics, vol. 121, Academic Press, Inc.,
  Orlando, FL, 1986. \MR{860042}

\bibitem[Voe03a]{VoeZ2}
Vladimir Voevodsky, \emph{Motivic cohomology with {${\bf Z}/2$}-coefficients},
  Publ. Math. Inst. Hautes \'Etudes Sci. (2003), no.~98, 59--104. \MR{2031199}

\bibitem[Voe03b]{Voered}
\bysame, \emph{Reduced power operations in motivic cohomology}, Publ. Math.
  Inst. Hautes \'Etudes Sci. (2003), no.~98, 1--57. \MR{2031198}

\bibitem[Voe10]{VoeEM}
\bysame, \emph{Motivic {E}ilenberg-{M}aclane spaces}, Publ. Math. Inst. Hautes
  \'Etudes Sci. (2010), no.~112, 1--99. \MR{2737977}

\bibitem[Wil75]{Wilson75}
W.~Stephen Wilson, \emph{The {$\Omega $}-spectrum for {B}rown-{P}eterson
  cohomology. {II}}, Amer. J. Math. \textbf{97} (1975), 101--123. \MR{0383390}

\end{thebibliography}

\end{document}